\documentclass[a4paper]{article}
\usepackage[utf8]{inputenc}%
\usepackage[american]{babel} 
\usepackage{lmodern}
\usepackage[T1]{fontenc}
\usepackage[style=numeric]{biblatex}
\usepackage{csquotes}
\usepackage[affil-it]{authblk}
\usepackage{xcolor}

\usepackage{hyperref}

\usepackage{amsmath,amssymb,amsthm,bbm}

\usepackage{tikz}
\usepackage{mathtools}
\usetikzlibrary{calc}

\usepackage{newunicodechar}
\newunicodechar{‑}{\nobreakdash-\hspace{0pt}}

\newcommand{\abs}[1]{\left\lvert#1\right\rvert}
\newcommand{\norm}[1]{\left\lVert#1\right\rVert}

\newcommand{\Z}{\mathbb Z}
 \newcommand{\R}{\mathbb R}
\newcommand{\C}{\mathbb C}

 \newcommand{\RR}{\mathbb R}
\newcommand{\CC}{\mathbb C}

\newcommand{\id}{\mathbbm1} \newcommand{\mathopf}[1]{{\operatorfont
    #1}}
\newcommand{\ee}{\mathopf e}

\newcommand{\SU}{\mathopf{SU}}
\newcommand{\U}{\mathopf{U}}

\newcommand{\su}{\mathfrak{su}} \newcommand{\LSL}{\mathfrak{sl}}
\newcommand{\diag}{\mathopf{diag}}

\let\setthat=\sett \newcommand{\quot}[3][\!]{\mathchoice%
  {\left.\raisebox{.2em}{$\displaystyle#2$}#1\middle/#1%
      \raisebox{-.2em}{$\displaystyle#3$}\right.}%
  {\left.\raisebox{.2em}{$#2$}#1\middle/#1%
      \raisebox{-.2em}{$#3$}\right.}%
  {\left.\raisebox{.08em}{$\scriptstyle#2$}#1\middle/#1%
      \raisebox{-.08em}{$\scriptstyle#3$}\right.}%
  {\left.\raisebox{.05em}{$\scriptscriptstyle#2$}#1\middle/#1%
      \raisebox{-.05em}{$\scriptscriptstyle#3$}\right.}%
}

\newcommand{\eps}{\varepsilon}

 \newcommand{\End}{\mathopf{End}}
 
\newcommand{\real}{\mathopf{Re}} 
\newcommand{\LG}{\mathfrak{g}} \newcommand{\LT}{\mathfrak{t}}
 \newcommand{\h}{h}
\newcommand{\dist}{\mathopf{dist}}

\DeclareMathOperator{\Ad}{Ad} 
 
\DeclareMathOperator{\ann}{ann} \DeclareMathOperator{\aff}{aff}
 
 \DeclareMathOperator{\rank}{rank}

\newcommand{\J}{\mathbf{J}} \newcommand{\I}{\mathopf{i}}
\newcommand{\e}{\mathopf{e}} 
\newcommand{\der}{\mathopf{d}} 
\newcommand\aug{\fboxsep=-\fboxrule\!\!\!\fbox{\strut}\!\!\!}

\newtheorem{theorem}{Theorem}[section]
\newtheorem{lemma}[theorem]{Lemma}
\newtheorem{corollary}[theorem]{Corollary}

\theoremstyle{definition} \newtheorem{definition}[theorem]{Definition}
\newtheorem{example}[theorem]{Example}
\newtheorem{remark}[theorem]{Remark}
\title{$G$-isotropy of $T$-relative equilibria within manifolds tangent to spaces with
  linearly independent weights}
\author{Mara Sommerfeld
  \\\href{mailto:mara.sommerfeld@uni-hamburg.de}{mara.sommerfeld@uni-hamburg.de}
\\\href{https://orcid.org/0000-0001-6194-8274}{https://orcid.org/0000-0001-6194-8274}}
\affil{Department of Mathematics, Universit\"at Hamburg,\\ Bundesstraße
  55, 20146 Hamburg, Germany}
\date{}
\begin{document}
\maketitle
\begin{abstract}
  We investigate the generic local structure of relative equilibria in
  Hamiltonian systems with symmetry $G$ near a completely symmetric
  equilibrium, where $G$ is compact and connected. Fix a maximal torus
  $T \subset G$ and identify the equilibrium with the origin within a
  symplectic representation of $G$. By a previous result, generically,
  for each $\xi \in \LT$ such that $V_0:= \ker \der^2(h-\J^\xi)(0)$
  has linearly independent weights, there is a manifold tangent
  to $V_0$ that consists of relative equilibria with generators in
  $\LT$. Here we determine their isotropy with respect to $G$. The
  main result asserts that for each of these manifolds of $T$-relative
  equilibria, there is a
  local diffeomorphism to its tangent space at $0$ that preserves the
  isotropy groups. We will then deduce that the $G$-orbit of the union
  of these manifolds is stratified by isotropy type. The stratum given
  by the relative equilibria of type $(H)$ has the dimension $\dim G
  -\dim H + \dim (\LT')^L$, where $\LT' \subset \LT$ is the orthogonal
  complement of $\mathfrak h \cap \LT$ and $L$ is the minimal adjoint
  isotropy subgroup of an element of $\LT$ with $H \subset L$. In the
  end, we consider some examples of these manifolds of $T$-relative
  equilibria that contain points with the same isotropy
  type with respect to the $T$-action but different isotropy type with
  respect to $G$.
\end{abstract}

\section*{Introduction}
Symmetries play a central role in the description of physical
systems. In classical mechanics, we usually consider actions with
compact isotropy groups of rank 1, but in non-classical physics we
often have to deal with groups of higher rank. For instance in gauge
theories, the symmetries of the vacuum state are typically products of
special unitary groups.

In my thesis \cite{thesis}, I investigated classical
finite-dimensional Hamiltonian systems with proper group actions, but
with connected isotropy groups of rank greater than 1. Even though
physicists usually consider infinite dimensional systems, this could
be a step to a better understanding of the dynamics caused by high
dimensional symmetries.

The main objects of interest of \cite{thesis} are \emph{relative
  equilibria}: states whose trajectories are contained in their group
orbit. In classical mechanics, they often correspond to solutions of
constant shape. If the physical system is modeled by a gauge system,
relative equilibria describe stationary physical states.

One of the main results of \cite{thesis} is of local nature: We
consider an isolated equilibrium with connected isotropy group $G$ and
investigate the generic structure of relative equilibria in a small
neighborhood of the equilibrium. If the equilibrium is a minimum of
the Hamiltonian, this means that we only consider low energy
states. Since we are only interested in the local structure, it
suffices to consider a $G$-representation $V$ with a $G$-invariant
symplectic form and a $G$-invariant Hamiltonian function $h$ with a
critical point at the origin.

The symplectic form induces a $G$-invariant complex structure on the
center space $V^c$ of $\der X_h(0)$, which is considered as a complex
representation.

The approach to find relative equilibria is to fix a maximal torus $T$
of $G$ and to search for invariant $T$-orbits. Due to the conjugation
theorem, their $G$-orbits form the set of relative equilibria with
respect to the $G$-action.

As discussed in \cite[Section 5.2.2]{thesis}, the equivariant
Weinstein-Moser theorem, introduced by Montaldi, Roberts, and Stewart
\cite{equivweinsteinmoser} already implies that generically for each
irreducible subrepresentation $W$ of $V^c$ and each weight that occurs
with multiplicity 1 in $W$ there is a local manifold of invariant
$T$-orbits that is at the origin tangent to the corresponding weight
space in $W$. Alternatively this follows from a theorem by Ortega and
Ratiu \cite[Theorem 4.1]{Ortega2004RelativeEN}.

The main observation in \cite{thesis} is that, in the case
$\rank(G)>1$, also for some particular sums of weight spaces there are
manifolds of invariant $T$-orbits whose tangent spaces at the origin
coincide with these sums: Generically, we obtain such a manifold for
each kernel of the form $\ker \der^2 (h-\J^\xi)(0)$ that has only
linearly independent weights that occur with multiplicity 1. (Here
$\J^\xi$ denotes the momentum map $\J$ evaluated at $\xi \in \LT$.)
Moreover, the union of these manifolds forms a Whitney-stratified
set. The proof of this statement is based on an application of
equivariant transversality theory, which was developed independently
by Bierstone \cite{Bierstone1977GeneralPO,Bierstone1977GenericEM} and
Field \cite{Field1977StratificationsOE, Field1977TransversalityI} in
the 70's of the last century. The method is very similar to Field and
Richardson's approach to equivariant bifurcation theory (see
\cite{Field1989SymmetryBA, Field1991LocalSO, field1996symmetry}).

This article deals with the computation of the isotropy types of these
relative equilibria. The isotropy type is not only an interesting
property of the solution but also gives information about the overall
structure of relative equilibria, since the $G$-orbits of these
manifolds consist of relative equilibria.

Moreover, we will need to compute the isotropy types of the generators
and momenta, which also play a key role in the description of the
structure of relative equilibria in Hamiltonian systems: Patrick and
Roberts \cite{Patrick2000TheTR} have shown that the relative
equilibria of Hamiltonian systems with a compact connected symmetry
group $G$ that acts freely are generically stratified by the isotropy
type $(K):=(G_\mu\cap G_\xi)$ of their momentum generator pair
$(\mu, \xi)$ with respect to the sum of the coadjoint and adjoint
action on $\LG^*\oplus \LG$. The corresponding stratum has the
dimension $\dim G + 2\dim Z(K) -\dim K$, where $Z(K)$ is the center of
$K$.  To prove this, they define a generic transversality
condition. As pointed out in \cite{thesis}, this condition is a
special case of a transversality condition that is natural in terms of
equivariant transversality theory and is also valid for non-free
actions. This implies that Patrick and Robert's result generically
holds for the action of the group $\left(\quot{N(H)}{H}\right)^\circ$
on the set of points with isotropy group $H$.

We will see that an application of this theory shows that the
$G$-orbit of the set formed by the manifolds of $T$-relative
equilibria tangent to the
kernels of the form $\ker \der^2(h-\J^\xi)(0)$, $\xi \in \LT$, with
linearly independent weights is
stratified by isotropy type. If $H \subset G$ is an isotropy subgroup,
then the stratum of the relative equilibria with isotropy type $(H)$
has the dimension
\begin{equation*}
  \dim G - \dim H + \dim(\LT')^L,
\end{equation*}
where $\LT'$ denotes the orthogonal complement of $\LT \cap \mathfrak
h$ in $\LT$ and $L = L(H)$ is the minimal isotropy subgroup of an
element of $\LT$ with respect to the adjoint action that contains $H$.

The main result to determine the isotropy subgroups asserts that for each of
these manifolds there is a diffeomorphism to a neighborhood of $0$
within its tangent space at the origin which preserves the isotropy
groups. Thus the problem is reduced to a problem in representation
theory.

This justifies the approach of \cite{thesis}: The main idea is to
investigate the structure of the $T$-relative equilibria since they
are contained in every $G$-orbit of relative equilibria. At first
glance, the method seems to have a drawback since we forget
information about the $G$-action in between and perform a
Lyapunov-Schmidt reduction with only $T$-invariant spaces. The results
presented in this article show that anyhow the $G$-isotropy groups are
preserved.

In the end, we will also illustrate an approach to compute the
isotropy types on the Lie algebra level in the relevant subspaces of a
given $G$-representation and discuss some examples. A striking
observation is that points of the same stratum of $T$-relative
equilibria can have different isotropy types with respect to the
$G$-action (so the $G$-orbit of the stratum is not necessarily a
manifold) and their isotropy groups can even be greater than the
isotropy groups of the corresponding weight spaces. (If we call
the weight vectors \emph{pure states} and points with nonzero
components in weight spaces corresponding to different weight
\emph{mixed states}, then the mixed states can have more symmetry than
each of their pure components).

\section{Preliminaries}

\subsection{Notation corresponding to the group action}
Throughout, let $G$ denote a compact connected Lie group and $\LG$ be its
Lie algebra. If $P$ is a space with a $G$-action and $p \in P$, then
the \emph{isotropy subgroup} of $p$ is given by
$G_p = \setthat{g \in G}{gp = p}$. Its conjugacy class is the
\emph{isotropy type} of $p$, denoted by $(G_p)$. The Lie algebra
$\LG_p$ of $G_p$ is called the \emph{isotropy Lie algebra} of $p$. Two
points $p$ and $q$ have the same \emph{type on the Lie algebra level}
if there is $g \in G$ with $\Ad_g(\LG_p)= \LG_q$, where $\Ad$ denotes
the adjoint action of $G$ on $\LG$, given by
  \begin{equation*}
     \Ad_gx := \frac \der {\der t} g\exp (t\xi)g^{-1}\big |_{t=0}.
  \end{equation*}

  For a subgroup $H \subset G$ and a subset $M \subset P$, the set of
  fixed points of $H$ within $M$ is denoted by
  $M^H = \setthat{m \in M}{hm =m \; \forall h \in H}$, the subset of
  points of isotropy type $(H)$ by
  $M_{(H)}=\setthat{m \in M}{G_m \in (H)}$ and the subset of points
  with isotropy subgroup $H$ by $M_H = M^H \cap M_{(H)}$.

\subsection{Notation and facts from representation theory}
Here we provide basic definitions and facts from the representation
theory of compact connected Lie groups, for details see for instance
\cite{btd} or \cite{hall2003lie}.

We will later on
regard symplectic representations as complex representations. Thus we
consider complex representations in this section.

A key idea in representation theory and as well, in our approach to
find relative equilibria, is to fix a maximal torus $T \subset G$ and
to consider a given finite-dimensional complex representation $V$ as a
$T$-representation.

Every group element is contained in a maximal torus, and all maximal
tori are conjugated to $T$. Correspondingly, each element of the Lie algebra
$\LG$ of $G$ is contained in the Lie algebra of a maximal torus, and
every Lie algebra of a maximal torus coincides with the image of the
Lie algebra $\LT$ of $T$ under the adjoint action of an element
$g \in G$.

As a $T$-representation, $V$ splits into the $T$-irreducible
subrepresentations, the \emph{weight spaces}. Each weight space is
1-dimensional and of the following form: Consider the exponential map
$\exp: \LT \to T$. If we use the identification $T \cong \quot
{\R^n}{\Z^n}$, then $\exp$ coincides with the projection $\R^n \to
\quot{\R^n}{\Z^n}$. An element $\alpha \in \LT^*$ is called an
\emph{integral form} iff $\alpha$ maps $\ker(\exp)$ to $\Z$. (When
we identify $\LT^*\cong (\R^n)^*$ with $\R^n$ via the standard inner
product, then the set of integral forms coincides with $\Z^n$.) Every
integral form $\alpha$ defines a $T$-representation on $\C$, which we
denote by $\C_\alpha$, via
\begin{equation*}
  T\ni \exp(\xi) \mapsto \ee^{2\pi \I \alpha(\xi)} \in \U(1).
\end{equation*}
These are exactly the irreducible complex $T$-representations. If
$\C_\alpha$ is a weight space of $V$, then $\alpha$ is the
corresponding \emph{weight}. Any non-zero element of $\C_\alpha$ is
called a \emph{weight vector} corresponding to $\alpha$.  (Note that
this definition of weights equates to \emph{real infinitesimal
  weights} or \emph{real weights} in \cite{btd} and
\cite{hall2003lie}.)

Of particular importance are the weights of the complexified adjoint
action
\begin{align*}
  G \times (\LG \otimes \C) &\to \LG \otimes \C\\
  (g, \xi\otimes z) &\mapsto (\Ad_g\xi)\otimes z,\\
\end{align*}
which are called \emph{roots}. Accordingly, the weight spaces and
weight vectors of $\LG\otimes \C$ are called \emph{root spaces} and \emph{root
  vectors} respectively. The roots occur in pairs
$\pm \rho \in \LT^*$.  As in \cite{btd}, we denote the root space
corresponding to $\rho$ by $L_\rho$ and the space
$(L_\rho\oplus L_{-\rho}) \cap \LG$ corresponding to $\pm \rho$ within
the real Lie algebra $\LG$ by $M_\rho = M_{-\rho}$. If $R^+$ is a
subset of the set of roots $R$ that contains exactly one element of
each pair $\pm \rho$, then
\begin{equation*}
  \LG = \LT \oplus \bigoplus_{\rho \in R^+} M_\rho.
\end{equation*}
The kernels of the roots are hyperplanes in $\LT$, called \emph{Weyl
  walls}. The connected components of
$\LT \setminus \left(\bigcup U_{\rho \in R^+} \ker \rho\right)$ are
called \emph{Weyl chambers}, their closures \emph{closed Weyl
  chambers}.  The orthogonal reflections about the Weyl walls with
respect to an adjoint-invariant inner product on $\LG$ restricted to
$\LT$ generate the Weyl group $W = W(G):= \quot {N(T)}{T}$. The
$W$-action on $\LT$ coincides with the one which is induced by the
adjoint action restricted to $N(T) \times \LT$, see \cite[Chapter V,
(2.19)]{btd}.

Each Weyl group orbit intersects every closed Weyl chamber in exactly
one point.  The weights of a $G$-representation $V$ are
$W$-invariant. If we fix a particular closed Weyl chamber $C$, each
Weyl group orbit of weights of $V$ has hence a unique element in $C$.
We call a weight $\alpha \in C$ \emph{higher} than a weight
$\beta \in C$ iff $\beta$ is contained in the convex hull of the Weyl
group orbit of $\alpha$. Then the irreducible $G$-representations are
in $1$-to-$1$-correspondence with the integral elements $\lambda$ of
$C$: For each $\lambda$ there is a unique irreducible finite
dimensional complex $G$-representation with highest weight $\lambda$.
Moreover, the Weyl walls determine the adjoint isotropy subgroups of
the elements of $\LT$:
\begin{lemma}\label{lem:adjoint_isotropy_connected}
  Isotropy subgroups with respect to the adjoint action are connected.
\end{lemma}
\begin{proof}
  Consider $\xi \in \LG$. By \cite[Chapter IV, (2.3)(ii)]{btd} the
  adjoint isotropy group $G_\xi$ coincides with the union of the
  maximal tori that contain $\xi$.
\end{proof}
Thus, the adjoint isotropy groups are determined by their Lie algebras.
\begin{lemma}[Infinitesimal formulation of {\cite[Chapter
    V,(2.3)(ii)]{btd}}] \label{lem:BTD}
  The adjoint isotropy Lie algebra of $\xi \in \LT$ is given by
  \begin{equation*}
    \LG_\xi = \LT \oplus \bigoplus_{\rho \in R^+, \xi \in \ker \rho} M_\rho.
  \end{equation*}
\end{lemma}
In other words: The isotropy subspaces of the adjoint action
intersected with $\LT$ are given by the intersections of Weyl walls.
\begin{corollary}\label{cor:intersection-groups}
  Let $K$ and $L$ be isotropy subgroups of points in $\LT$ with
  respect to the adjoint action, and let $\LT^K$ and $\LT^L$ be their
  fixed point subspaces within $\LT$ respectively. Then $K\cap L$ is
  also an isotropy of the adjoint action and its fixed point subspace
  within $\LT$ is given by the intersection of $\LT$ and all Weyl walls that
  contain both $\LT^K$ and $\LT^L$.
\end{corollary}
\begin{proof}
 Choose $\xi_K, \xi_L \in \LT$ with $G_{\xi_K}= K$ and $G_{\xi_L} =
 L$. Choose $\eps > 0$ small enough such that $\xi:=\xi_K+ \eps \xi_L$ is
 not contained in any Weyl wall that does not contain $\xi_K$. Then
 $\xi$ is contained exactly in the Weyl walls that
 contain $\xi_K$ and $\xi_L$. Thus by Lemma~\ref{lem:BTD} $\LG_\xi =
 \mathfrak k \cap \mathfrak l$, where $ \mathfrak k$ and $\mathfrak l$
 denote the Lie algebras of $K$ and $L$ respectively. Since $G_\xi$ is
 connected, we obtain
 \begin{equation*}
   G_\xi = G_\xi^\circ = (K\cap L)^\circ \subset K\cap L.
 \end{equation*}
 Conversely, we obviously have $K \cap L \subset G_\xi$.
\end{proof}

The roots also contain important information about the $G$-action on
the representation $V$: The $G$-action defines a Lie algebra action of
$\LG$ on $V$, that is a Lie algebra homomorphism $\LG \to \End(V)$,
given by
\begin{equation*}
  \xi x := \frac {\der}{\der t} \exp(t\xi) x \big|_{t=0}.
\end{equation*}
Since $V$ is a complex representation, this induces an action of the
complexified Lie algebra $\LG \otimes \C$ by $(\xi \otimes z)x:= z\xi
x$.

Note that $x \in V$ is a weight vector corresponding to the weight
$\alpha$ iff it holds for every $\xi \in \LT$ that
\begin{equation*}
  \xi x = (\xi\otimes 1)x= 2\pi \I\alpha(\xi) x.
\end{equation*}
\begin{lemma}
  If $Z \in \LG \otimes \C$ is a root vector corresponding to $\rho$
  and $x \in V$ is a weight vector corresponding to $\alpha$, then
  $Zx$ is either $0$ or a weight vector with weight $\alpha + \rho$.
\end{lemma}
\begin{proof} If $\xi \in \LT$, then
  \begin{equation*}
    \xi(Zx) =[\xi, Z]x + Z(\xi x) = 2\pi \I \rho(\xi) Z x + Z(2\pi \I
    \alpha(\xi) x) = (2\pi \I (\alpha+\rho)(\xi)) (Z x).
  \end{equation*}
\end{proof}
Note also that
$\LG \otimes \C = \LT \otimes \C \oplus \bigoplus_{\rho \in R}
L_\rho$. Thus if we consider a particular affine subspace $A$ of
$\LT^*$ whose underlying subspace is given by the span of the roots,
then the sum of the weight spaces corresponding to the weights
contained in $A$ is $\LG \otimes \C$-invariant and hence
$G$-invariant. We will use this fact below.

More precisely, the weight structure of an irreducible representation
is as follows:
\begin{theorem}[{\cite[Theorem 10.1]{hall2003lie}}]\label{thm:hall-weight-structure}
  If $V_\lambda$ is an irreducible finite-dimensional complex
  $G$-representation with highest weight $\lambda$, then an integral
  element $\alpha \in \LT^*$ is a weight of $V_\lambda$ iff $\alpha$
  is contained in the convex hull of $W \lambda$ and $\lambda -\alpha$
  is an integer combination of roots of $G$.
\end{theorem}
(Note that the weights can have a higher multiplicity than 1. The
highest weight always has multiplicity 1. The multiplicities of the
other weights are for instance determined by Kostant's multiplicity
formula, see \cite[Chapter VI, (3.2) and (3.3)]{btd}.)

In the applications, we will also consider the coadjoint
representation, the dual of the adjoint representation. Choosing an
adjoint invariant inner product, we can identify both
representations and accordingly $\LT$ and $\LT^*$. This defines the
Weyl walls in $\LT^*$. Since we will often consider $G_\mu$-invariant
subrepresentations, where $\mu$ is a momentum, which is an element of
$\LT^*$, we formulate the following results in the coadjoint version.

\begin{lemma}\label{lem:orthogonal_affine}
  Let $V$ be a complex $G$-representation. For any $\mu \in \LT^*$,
  let $(\LT^*)^\mu$ denote the intersection of $\LT^*$ and all Weyl
  walls in $\LT^*$ that contain $\mu \in \LT^*$.  Let $A$ be the set
  of weights contained in a particular affine subspace that is a shift
  of the orthogonal complement of $(\LT^*)^\mu$.  Then the sum of the
  corresponding weight spaces $\bigoplus_{\alpha\in A}\CC_\alpha$ is
  $G_\mu$-invariant.
\end{lemma}
\begin{remark}
  Note that $(\LT^*)^\mu = \LT^*$ iff $\mu$ is not contained in any
  Weyl wall. In this case, affine subsets are orthogonal to
  $(\LT^*)^\mu$ iff they consist of a single point. Thus $A$ contains
  at most one weight.
\end{remark}
\begin{proof}[Proof of Lemma~\ref{lem:orthogonal_affine}.]
  We identify $\LG^*$ and $\LG$ via $G$-invariant product and consider
  $\mu$ as an element of $\LT$.

  By Lemma~\ref{lem:BTD}, $\LG_\mu\otimes\C = \LT\otimes \C \oplus
  \bigoplus_{\alpha \in N}\CC_\alpha$, with $N=\setthat{\alpha \in
    R}{\mu \in \ker \alpha}$. Thus $N$ coincides with the set of roots of
  $G_\mu$. Since $\alpha$ is orthogonal to the Weyl wall $\ker \alpha$
  (with the respect to the above identification of $\LG$ and $\LG^*)$,
  all elements of $N$ are orthogonal to $(\LT^*)^\mu$.

  By Theorem~\ref{thm:hall-weight-structure} the weights of any
  irreducible $G_\mu$-representation differ by a sum of elements of
  $N$. Thus the weights of any irreducible $G_\mu$-subrepresentation
  of $V$ are contained in the same affine space orthogonal to
  $(\LT^*)^\mu$.
\end{proof}

\subsection{The group $\SU(n)$: Maximal torus, weights, root system,
  Weyl group}\label{sec:su3}
Here we summarize the basic data about the groups $\SU(n)$ that we need
for our examples in the end of the article. Again, we refer to
\cite{btd} and \cite{hall2003lie} for details.

A maximal torus $T$ of $\SU(n)$, $n\ge 2$, is given by the subgroup
formed by the diagonal matrices. The diagonal entries are of the form
$\e^{2\pi \I r_i}$, $r_i \in \R$, $i = 1, \dots, n$. Since the product
of the diagonal entries is 1, we can assume that
$\sum_{i=1}^n r_i= 0$. Hence the Lie algebra $\LT$ can be identified
with the $(n-1)$-dimensional subspace of elements of $\R^n$ whose
entries sum up to $0$. The Weyl group is given by the group $S_n$
which acts on $\LT$ by permuting the $n$ entries.

Alternatively, we fix the basis given by the $(n-1)$-vectors of the
form $(0, \dots, 0, 1, -1, 0, \dots)$ and denote the corresponding
coefficients.

For each of these two representations, the kernel of the exponential
map $\exp: \LT \to T$ is given by the
vectors with integer entries.

Both representations have their advantages: The second one is more
concise, but the symmetry is easier to see in the first one.

Similarly, we obtain
$\LT^*\cong \quot{\R^n}{\langle (1, 1, \dots, 1)\rangle}$. We will
switch between different representations of $\LT^*$: If $n$ is large,
we choose the representative in $\R^n$ whose entries sum up to
$0$. Br\"ocker and tom Dieck \cite{btd} take the representative with
last entry $0$, so they only need to consider $n-1$ entries, but we do
not use this description here. Instead, if $n$ is small enough to get
the symmetry from the graphics, we denote the coefficients with respect
to the dual base of the vectors $(0, \dots, 0, 1, -1, 0, \dots)$.
Note that
the integral forms correspond to the vectors with integer entries with
respect to the second and the third representation of $\LT^*$, but not for the
first one. A vector $(\alpha_1, \dots, \alpha_n)$ with $\sum_i \alpha_i
= 0$ represents an integral form iff the numbers $n \alpha_i$ are
integers that are congruent modulo $n$.

As vectors in $\R^n$ with entry sum $0$, the roots of $\SU(n)$ are given
by the vectors with one entry $1$, one entry $-1$ and $n-2$ entries
$0$. In the classification of root systems, this is the root system of
type $A_{n-1}$. The root system of a product of groups is given by a
product of the corresponding root systems, see \cite{btd}.

\subsection{Hamiltonian relative equilibria}
In this article, we investigate the local structure of relative
equilibria in Hamiltonian systems with $G$-symmetry near a given
equilibrium with isotropy group $G$.

A point $p$ of a $G$-manifold $P$ is a \emph{$G$-relative equilibrium}
of a $G$-equivariant vector field $X$ iff the group orbit $Gp$ is
$X$-invariant. Equivalently, there exists $\xi \in \LG$ with
$X_p = \xi p$. Then $\xi$ is called a \emph{generator} (or
\emph{velocity}) of the relative equilibrium $p$. If $\xi$ is a
generator of $p$, then so is $\xi + \eta$ for any $\eta \in
\LG_p$. Thus the generator is not unique in general, but it is unique
regarded as an element of $\quot \LG {\LG_p}$.

Now we suppose that $P$ is a smooth symplectic manifold with a
$G$-invariant symplectic form $\omega$ and consider the Hamiltonian
vector field $X_h$ of a smooth $G$-invariant function $h: P \to \R$,
i. e.
\begin{equation*}
  \der h(x) = \omega(x)(X_h(x), \cdot)
\end{equation*}
for every $x \in P$. Then there is an additional
structure: We assume that there is a \emph{momentum map} $\J: P \to
\LG^*$ which satisfies
\begin{equation*}
  \langle \der \J(x) v, \xi\rangle = \omega(x) (\xi x,v)
\end{equation*}
and is equivariant with respect to the coadjoint action on
$\LG^*$. (In general, the momentum map exists at least locally, which
suffices for our purpose, see \cite[Part II, Chapter
26]{guillemin1984symplectic}, also for sufficient conditions for the
existence of a global momentum map.)

Thus the vector fields $\xi_P\colon x \mapsto \xi x$ are Hamiltonian:
Let $\J^\xi:P \to \R$ denote the function $x \mapsto \langle \J(x),
\xi \rangle$. Then $\J^\xi$ is the Hamiltonian function corresponding
to $\xi_P$. Therefore we have another equivalent characterization of
relative equilibria: $p \in P$ is a relative equilibrium with
generator $\xi$ iff $p$ is a critical point of the function $h-\J^\xi$.

Note that the functions $\J^\xi$ are $G_\xi$-invariant but in general
not $G$-invariant.

The momentum is constant along trajectories of the Hamiltonian vector
field. This implies in particular that for a relative equilibrium $p$
with generator $\xi$ and momentum $\mu := \J(p)$, we have $\xi \in
\LG_\mu$. Thus if we identify $\LG$ and $\LG^*$ via a $G$-invariant
product, then $[\xi, \mu] = 0$.

By the equivariant Darboux theorem (see \cite[Theorem
22.2]{guillemin1984symplectic} and note also the correction by
Dellnitz and Melbourne \cite{equivariantdarboux}), a point with
isotropy group $G$ in a symplectic manifold with a $G$-invariant
symplectic form has a neighborhood such that there is a
$G$-equivariant symplectomorphism to a neighborhood of $0$ within a
\emph{$G$-symplectic representation} $V$, i. e. a $G$-representation
$V$ together with a $G$-invariant symplectic form. Thus we will
restrict ourselves to this case. Then the functions $\J^\xi$ are
quadratic forms on $V$.

\subsection{Review of the results of my thesis}
Here I provide main ideas and results of my thesis \cite{thesis} and
point out some corrections.

Let $V$ be a $G$-symplectic representation and $h:V \to \R$ be a smooth
$G$-invariant Hamiltonian function. Then in \cite{thesis}, I
follow the approach by Ortega and Ratiu \cite{Ortega2004RelativeEN} to
find pairs $(x,\xi) \in V \times \LG$ such that $x$ is a critical
point of $h-\J^\xi$:

Consider the values of $\xi \in \LG$, such that
$V_0:=\ker \der^2(h-\J^\xi)(0)$ is non-trivial. Then perform a
Lyapunov-Schmidt reduction to obtain a function
$g: V_0 \times \LG \to \R$ such that for each $\eta$ in a neighborhood
of $\xi$, the critical points of $g(\cdot, \xi)$ near $0 \in V_0$ are
in bijection with the critical points of $h-\J^\xi$ in $V$.

For bifurcation problems with symmetry but a trivial action on the
parameter space, the Lyapunov-Schmidt-reduction can be performed in a
way that the symmetry is preserved. Here, our parameter space is the
Lie algebra $\LG$ and the $G$-action on $\LG$ is the adjoint action,
which is non-trivial in general. The way out is to consider a smaller
parameter space: Ortega and Ratiu \cite{Ortega2004RelativeEN} only
search for solutions in $\LG^{G_\xi}$.

The approach of \cite{thesis} is to consider the action of the maximal
torus $T$ and to search for the $T$-relative equilibria. These are
exactly the $G$-relative equilibria that have a generator in
$\LT$. The adjoint action of $T$ on $\LT$ is trivial and hence the
Lyapunov-Schmidt-reduction preserves the $T$-symmetry.  Since every
$\xi \in \LG$ has an element of $\LT$ in its adjoint orbit and since
$p$ is a relative equilibrium with generator $\xi$ iff $gp$ is a
relative equilibrium with generator $\Ad_g \xi$, the union of
$G$-orbits of $T$-relative equilibria coincides with the set of
$G$-relative equilibria.

Another important tool is the linear theory developed by Melbourne and
Dellnitz \cite{melbourne_dellnitz_1993} together with a theorem about
generic bifurcation of Hamiltonian vector fields with symmetry by
these two authors and Marsden \cite[Theorem
3.1]{Dellnitz1992GenericBO}. It implies that generically, there is a
$G$-invariant inner product $\langle \cdot, \cdot \rangle$ on $V$ such
that $\omega = \langle \cdot, J \cdot \rangle$ and $J$ defines a
$G$-equivariant complex form on the \emph{center space} $V^c$ of
$\der X_h(0)$ that commutes with $\der X_h(0)|_{V^c}$. (The center
space $V^c$ is given by the sum of the generalized eigenspaces of
$\der X_h(0)$ corresponding to eigenvalues with real part $0$). It is
easy to see that $V_0 \subset V^c$ for every $V_0$ of the above
form. Thus we can consider $V_0$ as a complex $T$-representation if
$\xi \in \LT$. Hence $V_0$ is a sum of weight spaces of the complex
$G$-representation $V^c$. Each weight space $\C_\alpha$ is a
$T$-symplectic representation with symplectic form
$\langle \cdot, \I \cdot \rangle$, where
$\langle \cdot, \cdot \rangle$ is the real inner product described
above.

The existence of a $G$-invariant inner product of this kind follows
from the following genericity assumption:

\begin{definition}
  $h$ satisfied the \emph{generic center space condition (GC)} iff
  $\der X_h(0)$ is non-degenerate and $V^c$ splits into
  \emph{irreducible $G$-symplectic} subrepresentations ($G$-symplectic
  representations with no proper non-trivial $G$-symplectic
  subrepresentations) that coincide with the eigenspaces of
  $\der^2h(0)|_{V^c}$.
\end{definition}

Note that (GC) implies that there is a basis of weight vectors of
$V^c$, namely a union of weight vectors of the eigenspaces of
$\der^2h(0)|_{V^c}$, with respect to which $\der^2h(0)|_{V^c}$ is a
diagonal matrix. Moreover, with respect to this basis
$\der^2\J^\xi(0)|_{V^c}$ is also a diagonal matrix for every
$\xi \in \LT$: If $S^c$ denotes the set of weights of $V^c$ (counted
with multiplicities) and $x_\alpha$ denotes the $\C_\alpha$-component
of $x \in V^c \cong \bigoplus_{\alpha \in S^c} \C_\alpha$, then
$\J^\xi(x) = \pi \sum_{\alpha \in S^c} \abs{x_\alpha}^2\alpha$. Hence
$\der^2\J^\xi(0)|_{V^c}$ has the diagonal entries $2\pi
\alpha(\xi)$. Thus if $\C_\alpha$ is a weight space of the
$\der^2h(0)|_{V^c}$-eigenspace corresponding to the eigenvalue $c_i$,
then for any $\xi \in \LT$ the space $\C_\alpha$ is contained in
$\ker \der^2 (h-\J^\xi)(0)$ iff $2\pi \alpha(\xi) = c_i$.

From now on, we assume (GC) and in addition the following genericity
assumption (NR'). Together, they determine the weight structure of the
subspaces that occur as non-trivial kernels of the form
$\ker \der^2(h-\J^\xi)(0)$, $\xi \in \LT$:

\begin{definition}\label{def:full}
  Let $T$ be a real vector space and $S = \bigcup_{i=1}^n S_i$ be a
  union of subsets $S_i \subset T^*$. Then $S$ is \emph{full} iff for
  every vector $(c_1, \dots, c_n)\in \RR^n$, there is an $x \in T$
  with
  \begin{equation*}
    \forall i: \forall \alpha \in S_i: \alpha(x) =c_i. 
  \end{equation*}
\end{definition}

\begin{remark}\label{rem:full}
  Let $W_i$ denote the underlying subspace of the affine set
  $\aff(S_i)$ and let $W$ be the sum of the spaces $W_i$. Each set
  $S_i$ projects to a single point $s_i$ in $\quot {T^*}{W}$. Then $S$
  is full iff the set $\{s_i\}_{i\in I}$ (possibly with
  multiplicities) is linearly independent in
  $\quot {T^*}{W}$: Given $c = (c_1, \dots, c_n) \in \R^n$, let
  $X_c \subset T$ denote the solution set of the corresponding
  equation system. Identifying $T$ and $T^{**}$, each $x \in X_c$ is
  contained in $\ann(W)$. Thus $S$ is full iff the linear map
  $s\colon \ann(W) \to \R^n$ defined by
  $x \mapsto (s_1(x), \dots, s_n(x))$ has full rank.

  The dimension of $X_c$ coincides with the dimension of $\ker
  s$. Hence $\dim X_c = \dim \ann(W) -n$.

  (This alternative description of full sets is similar to the one
  given in \cite[Remark~6.60]{thesis}, but clearer.)
\end{remark}

\begin{definition}\label{def:NR'}
  Suppose that the $G$‑invariant Hamiltonian function $\h$ satisfies
  condition (GC). Then \emph{the generalized non-resonance condition
    (NR')} holds for $\h$ iff for each union $S = \bigcup_i S_i$ of
  sets $S_i$ of linearly independent weights of the eigenspaces $U_i$ of
  $\der^2\h(0)$ and the vector $c = (c_1, \dots c_n)$ of the
  corresponding eigenvalues, there is a $\xi \in \LT$ with
  \begin{equation*}
    \forall i: \forall \alpha \in S_i: \alpha(\xi) =c_i
  \end{equation*}
  iff $S$ is full.
\end{definition}

\begin{remark}
  If $G = T$ then (NR') is equivalent to the \emph{condition (NR)}
  that for each $\xi \in \LT$ the kernel $\ker \der^2(\h-\J^\xi)(0)$
  has only linearly independent weights.

  We obtain another generic requirement, which is more natural but
  slightly stronger, if we demand for unions $S = \bigcup_i S_i$ of
  sets $S_i$ of weights of the eigenspaces $U_i$ of $\der^2\h(0)$ in
  general that there is such a $\xi$ iff $S$ is full.
\end{remark}



One of the main results of \cite{thesis} is \cite[Lemma~6.55]{thesis}, which states
that if $V_0 = \ker \der^2 (h-\J^\xi)(0)$ for $\xi \in \LT$ has only
linearly independent weights, then there is a local manifold of
$T$-relative equilibria tangent to $V_0$. To apply \cite[Lemma~6.55]{thesis}
to $G$-representations, we only have to determine the kernels of this
kind with linearly independent weights. Here, the statement given in
Theorem 6.64 of \cite{thesis} is not correct in the case that $V^c$ is
a reducible $G$-representation. The corrected version is as follows
(the difference to the original version is pointed out below):

\begin{theorem}[Corrected version of Theorem 6.64 of
  \cite{thesis}.]\label{thm:main-thesis}
  Let $G$ be a connected compact Lie group with maximal torus $T$ and
  $V$ be a symplectic $G$‑representation. Suppose that $\h:V\to\RR$ is
  a smooth $G$‑invariant Hamiltonian function with critical point at
  $0$ that satisfies the genericity assumptions (GC) and (NR').

  Consider (possibly empty) subsets $S_i$ of weights of the
  eigenspaces $U_i$ of $\der^2h(0)|_{V^c}$ and set
  $S:=\bigcup_{i\in I} S_i$. Let $W_i$ denote the underlying subspace
  of the affine space $\aff (S_i)$ and $W$ be the sum of the spaces
  $W_i$, $i \in I$. If $S$ is linearly
  independent and in addition, each $S_i$ the maximal subset
  of weights of $U_i$ contained in the affine space  $\aff (S_i) + W$
  (in particular each weight of $S_i$ occurs with multiplicity 1 in
  $U_i$), then there is a $T$‑invariant manifold of $T$‑relative
  equilibria whose tangent space at $0$ is given by the sum of the
  corresponding weight spaces of the elements of $S$: For
  $\alpha \in S_i$, we obtain the summand $\CC_\alpha \subset U_i$.
\end{theorem}

\begin{remark}
  The condition given in the original version, \cite[Theorem
  6.64]{thesis}, is that each $S_i$ is maximal in $\aff(S_i)$. This is
  not strong enough:
  
  If
  $X\subset \LT$ denotes the set of elements $\xi \in \LT$ such that
  $\ker \der^2 (h-\J^\xi)(0)$ contains all the weight spaces
  corresponding to $S$, then $X$ coincides with the solution set of
  \begin{equation*}
    \forall i: \forall \alpha \in S_i: 2 \pi \alpha(\xi) = c_i.
  \end{equation*}
  Hence any element of $W$ is contained in the annihilator of
  $X$. Thus any weight $\alpha$ of $U_i$ contained in $\aff(S_i) + W$
  satisfies $ 2 \pi \alpha(\xi) = c_i$ for every $\xi \in X$, too. Thus if there is a
  $\xi \in \LT$ such that $\ker \der^2 (h-\J^{\xi})(0)$
  coincides with the sum of the weight spaces corresponding to $S$,
  then necessarily, for every $i$, the set $S_i$ is a maximal subset of
  the set of weights of $U_i$ within the space $\aff(S_i) + W$. The
  proof of Theorem~\ref{thm:main-thesis} shows that this condition is
  also sufficient.
\end{remark}
\begin{proof}[Proof of Theorem~\ref{thm:main-thesis}]
  We show  that there is indeed
  a $\xi\in \LT$ such that the space $\ker \der^2 (\h-\J^\xi)(0)$ consists of
  the corresponding weight spaces of the weights contained in $S$:
  Since the $S_i$ are linearly independent, 
  there are
  nonempty subsets $X_i \subset \LT$ such that
  $\bigoplus_{\alpha \in S_i} \CC_\alpha\subset U_i$ is contained in
  $\ker \der^2 (\h-\J^\xi)(0)$ if $\xi \in X_i$. Since
  $S= \bigcup_{i \in I} S_i$ is linearly independent, the intersection
  $X:=\bigcap_i X_i$ is non-empty. Moreover, there is a $\xi \in X$ with
  $\ker \der^2 (\h-\J^\xi)(0)= \bigoplus_{\alpha \in S} \CC_\alpha$:

  Suppose that we add a weight $\beta$ of $U_j$ to the set
  $S_j$. Let $\tilde X_j \subset X_j$ denote the intersection of $X_j$
  with the set of solutions $\xi$ of $2 \pi \beta(\xi) = c_j$ and set
  $\tilde X = X \cap \tilde X_j$. Then we have to show that $\dim
  \tilde X < \dim X$ or $\tilde X$ is empty.
  By assumption, $\beta$ is not contained in $\aff(S_j)$. Hence if
  $\tilde X$ is non-empty, $S_j \cup
  \{\beta\}$ is linearly independent and thus by
  condition~(NR'), after adding $\beta$ to $S_j$ we still have a full
  set. By assumption, $\beta -\alpha \notin W$ if $\alpha \in
  S_j$. Thus if $S_j$ is non-empty, the sum $\tilde W$ of the underlying subspaces of
  $\aff(S_i)$ for $i \ne j$ and $\aff(S_j \cup \{\beta\})$ is
  of higher dimension than $W$ and hence
  \begin{equation*}
    \dim \tilde X = \dim \ann(\tilde W) -n < \dim \ann(W) -n = \dim X,
  \end{equation*}
  where $n$ denotes the number of non-empty sets $S_i$.
  If $S_j$ is empty, then $\tilde W = W$ and
\begin{equation*}
    \dim \tilde X = \dim \ann(W) -(n+1) < \dim \ann(W) -n = \dim X.
  \end{equation*}
  Now the result follows from \cite[Lemma 6.55]{thesis}.
\end{proof}

As shown in \cite[Lemma 6.55]{thesis}, for each
$V_0 = \ker \der^2 (h-\J^\xi)(0)$ with linearly independent
weights, we obtain a manifold of $T$-relative equilibria tangent to
$V_0$ in $0$. The local manifold of relative
equilibria is given by the image of a local immersion $m_{V_0}$ of a
neighborhood of $0$ in $V_0$ into $V$ with $m_{V_0}(0)=0$. If
$W_0 \subset V_0$ is a $T$-invariant subspace, then the map $m_{W_0}$
locally coincides with the restriction of $m_{V_0}$ to (a neighborhood
of 0 within) $W_0$.
\begin{definition}
We refer to the image of $m_{V_0}$ as the \emph{manifold that
bifurcates at $V_0$} or at the set $X$ with $X = \setthat{\xi
\subset \LT}{V_0 \subset \ker \der^2(h-\J^\xi)(0)}$.
Similarly, the
set of points in the image of $m_{V_0}$ that are not contained in the
image of $m_{W_0}$ for any $T$-invariant proper subspace $W_0$ of $V_0$ is
called the \emph{stratum that bifurcates at $V_0$ }(or \emph{at $X$}).
\end{definition}

The idea of the proof is based on equivariant transversality theory.
The approach is very similar to an application of equivariant
transversality theory to bifurcation theory worked out by Field,
partly together with Richardson \cite{Field1989SymmetryBA,
  Field1991LocalSO, field1996symmetry}. See Field's book
\cite{field2007dynamics} for an introduction into the theory.

We consider the function $g: V_0 \times \LT \to \R$ obtained from
Lyapunov-Schmidt reduction. Up to third order, $g$ coincides with the
function
\begin{equation*}
  (x,\xi) \mapsto (h-\J^\xi)(x)
\end{equation*}
restricted to $V_0 \times
\LT$. If $V_0 \cong \bigoplus_{i=1}^l \C_{\alpha_i}$, then the
quadratic polynomials $p_i(x) = \abs{x_{\alpha_i}}^2$ for $x =
(x_{\alpha_1}, \dots, x_{\alpha_l})$ generate the
ring of invariant polynomials on $V_0$. Hence by invariant theory, the
$T$-invariant function $g$ is of the form $g(p_1, \dots, p_l,
\xi)$. Thus we can decompose the function $(x, \xi) \mapsto
\nabla_{V_0}g(x,\xi)$ into $g = \vartheta \circ \Gamma$, given by
\begin{align*}
  \Gamma:V_0 \times \LT &\to V_0 \times \R^l\\
  (x,\xi) &\mapsto \left(x, \partial_{p_1}g(x,\xi),
            \dots, \partial_{p_l}g(x,\xi)\right)
\end{align*}
and
\begin{align*}
  \vartheta: V_0 \times \R^l \to V_0\\
  (x,t) \mapsto \sum_{i=1}^lt_i\nabla p_i(x) = 2 \sum_{i=1}^lt_i x_{\alpha_i}
\end{align*}
(Here we choose the $T$-invariant real inner product $\langle
x,y\rangle = \real \sum_{i=1}^l x_{\alpha_i}\overline{y_{\alpha_i}}$
on $V_0$.)
In terms of equivariant transversality theory, $g$ is
\emph{$G$-1-jet-transverse} to $0 \in V_0$ iff $\Gamma$ is transverse
to the Whitney-stratified set $\Sigma:= \vartheta^{-1}(0)$ in
$0$. As shown in \cite{thesis}, $\Gamma$ is even transverse to $0 \in V_0
\times \R^l$. We then obtain for every $x$ in a neighborhood of the
origin of $V_0$ a generator $\xi(x) \in \LT$, unique modulo $\LT_x$, with
$\nabla_{V_0}g(x,\xi(x_0)) =0$. Then $m_{V_0}(x)$ is defined as the
corresponding critical point of $h-\J^{\xi(x)}$ in $V$.

More precisely, let $V_1$ denote the orthogonal complement of
$V_0$. From the Lyapunov-Schmidt reduction, we obtain a local
$T$-equivariant map
$v_1$ with image in $V_1$ defined in a neighborhood of $(0,
\xi_0)$, if $V_0 = \ker \der^2(h-\J^{\xi_0})(0)$, that solves
\begin{equation*}
  \der_{V_1} (h-\J^\xi)(x+v_1(x,\xi)) = 0
\end{equation*}
uniquely within some neighborhood of $0$. Then
$v_1(x, \xi + \eta) \in V_1$ is constant in $\eta \in \LT_x$ (see
below) and so is $g$, which is given by
$g(x, \xi) = (h-\J^\xi)(x+v_1(x,\xi))$. Thus
$m_{V_0}(x) = x+v_1(x,\xi(x))$ is well-defined, even though $\xi(x)$ is
an element of $\quot \LT {\LT_x}$. Moreover, if $V_0$ is contained in
$\tilde V_0 = \ker \der^2(h-\J^{\eta_0})$, the map $v_1$ defined for
pairs $(x, \xi)$ with $x \in V_0$ solves the defining equation of
the corresponding map for $\tilde V_0$ on its domain. Thus the
definition of the maps $v_1$ is compatible with restrictions, and
hence this holds for the embeddings $m_{V_0}$.

\begin{remark}\label{rem:constant_in_l}
  Since we will need the argument later on and also have to add some
  corrections, we now elaborate this point in more detail: 
  Let $T_0 \subset T$ be the kernel of the $T$-action on $V_0$. Its
  Lie-algebra $\LT_0$ is given by the common kernel of the
  weights of $V_0$. Since by $T$-equivariance of $v_1$, the point $x+
  v_1(x,\xi)$ is contained in $V^{T_0}$ for every $x \in V_0$ and $\xi \in
  \LT$, the derivative $\der_V\J^\eta(x +  v_1(x,\xi))$ vanishes for
  every $\eta \in \LT_0$. Thus $v_1(x, \xi)$ also solves
  \begin{equation*}
     \der_{V_1} (h-\J^{\xi+\eta})(x+v_1(x,\xi)) = 0
  \end{equation*}
  and by uniqueness $v_1(x,\xi) = v_1(x,\xi+\eta)$ if both are defined. 
  This implies that the map $v_1$ defined on a
  neighborhood of $(0,\xi_0)$ can be extended into a neighborhood
  which is invariant under addition of elements of $\LT_0$ such
  that $v_1$ is constant in the direction $\LT_0$. 
 (In principle,
  this argument is given in \cite[Remark 2.10]{thesis}, but it has some
  inaccuracies. In particular, the argument is given in the general,
  possibly non-abelian case there, but it lacks the condition that $L$ is a
  normal subgroup of $G$, which implies that $\LG^L + \mathfrak l =
  \LG$, see for instance, \cite[Remark 3.10.2]{field2007dynamics}. However, we
  only need the abelian case.)

  If $V_0 \subset \tilde V_0 = \ker \der^2(h-\J^{\eta_0})(0)$, then
  $\eta_0 - \xi_0 \in \LT_0$ and hence
  \begin{equation*}
    \tilde V_0^{T_0} \subset \ker \der^2(h-\J^{\xi_0})(0)= V_0,
  \end{equation*}
  thus $V_0 = \tilde
  V_0^{T_0}$. Moreover, we have $v_1(x, \xi) \in V_1^{T_0}$ for $x \in V_0,
  \xi \in \LT$. Let $\tilde V_1$ denote the orthogonal complement of
  $\tilde V_0$ and note that the $T$-invariance of the inner product implies
  \begin{equation*}
  V_0\oplus V_1^{T_0} = V^{T_0} = \tilde V_0^{T_0} \oplus \tilde
  V_1^{T_0} =  V_0\oplus \tilde V_1^{T_0}.
\end{equation*}
Together with $\tilde V_1 \subset V_1$, we obtain $V_1^{T_0} = \tilde
V_1^{T_0}$.
 As shown above, $v_1$ is defined on a neighborhood of $(0,\eta_0)$
 within $V_0\times \LT$. If $\tilde v_1$ is the corresponding map from
 a neighborhood of $(0,\eta_0)$ within $\tilde V_0\times \LT$ to
 $\tilde V_1$, then $v_1$ solves the defining equation of $\tilde v_1$
 on its domain. Thus $v_1 = \tilde v_1$ on the intersection of their
 domains.
\end{remark}

In this article, we will use the decomposition
$g = \vartheta \circ \Gamma$ again and consider the preimages under
$\vartheta$ of isotropy subspaces of $V_0$ with respect to the
$G$-action. This is one of the key arguments in the proof that --
despite the fact that the space $V_0$ is $T$-invariant but in general
not $G$-invariant -- the map $m_{V_0}$ preserves the $G$-isotropy
subgroups. We will show this in section~\ref{sec:reduction}.

The knowledge of the isotropy subgroups of the $T$-relative equilibria
is essential for the further investigation of their
$G$-orbits. However, a first step is already done in \cite{thesis}: We
can classify which of the $T$-relative equilibria in the images of
the maps $m_{V_0}$ are contained in the same $G$-orbit.

Recall that we say that the image of $m_{V_0}$ bifurcates at
\begin{equation*}
  X = \setthat{\xi \in \LT}{V_0 \subset \ker
  \der^2(h-\J^\xi)(0)}.
\end{equation*}
Suppose that $X$ is contained in a Weyl wall. This is the case iff one
of the subsets $S_i\subset S$ of weights of $V_0$ within an eigenspace $U_i$
of $\der^2 \h(0)$ contains a pair of weights
$\alpha \ne w \alpha$, where $w \in W$ denotes the reflection about
that Weyl wall. Then the image of $m_{V_0}$ lies in the $G$-orbit of the
image of $m_{W_0}$, where
$W_0 := \bigoplus_{S \setminus\{\alpha\}} \C_\alpha$, see
\cite[Theorem 6.74]{thesis}.  Hence, we only
have to consider sets of weights $S$ as in
Theorem~\ref{thm:main-thesis} that do not contain such a Weyl
reflection pair. Then the $G$-orbits of the corresponding strata
coincide for sets of weights $S, S'$ of this kind iff there is an
element $w' \in W$ with $S' =w'S$. Otherwise the $G$-orbits are
disjoint.

\section{Isotropy groups: Reduction to representation
  theory}\label{sec:reduction}
In the following, we will investigate the isotropy types of the
relative equilibria that exist by Theorem~\ref{thm:main-thesis} and
the isotropy types of their momenta and their generators. All three
contain important information about the structure of the set of
relative equilibria: As mentioned in the introduction, Patrick and
Roberts \cite{Patrick2000TheTR} have shown that if $G$ acts freely on
a symplectic manifold $(P,\omega)$ such that $\omega$ is
$G$-invariant, then generically the set of relative equilibria is
stratified by the conjugacy class $K$ of the intersection
$G_\xi \cap G_\mu$ of the isotropy subgroups of the generator $\xi$
and the momentum $\mu$. The dimension of the corresponding stratum is
given by $\dim G + 2 \dim Z(K)-\dim K$, where $Z(K)$ denotes the
center of $K$. As pointed out in \cite{thesis}, for a non-free action
and any isotropy subgroup $H$, Patrick and Robert's result generically
applies to the free action of the group
$\left( \quot {N(H)}{H}\right)^\circ$ on $P_H$, where $N(H)$ denotes
the normalizer of $H$ and $\left( \quot {N(H)}{H}\right)^\circ$ the
identity component of $\quot {N(H)}{H}$. We will use this in
Section~\ref{sec:dimensions} to compute the dimensions of the strata
within the $G$-orbits of the $T$-relative equilibria that exist by
Theorem~\ref{thm:main-thesis}.

Moreover, the investigation of the momenta, the generators and their
isotropy subgroups plays a key role in the calculation of the isotropy
types of the relative equilibria themselves.

In this section, we show that the local isotropy structure of the
manifolds of T-relative equilibria that exist by
Theorem~\ref{thm:main-thesis} coincides with that of their tangent
spaces at the origin. In addition, we link the isotropy groups of the
momenta and generators to those of the relative equilibria of the
linearized Hamiltonian vector field.

We start with an investigation of the momenta:

\subsection{Momenta}

Throughout, let $V$ be a $G$-symplectic representation. We consider the
smooth $G$-invariant Hamiltonian function $\h:V \to 0$ with critical
point at $0$ and assume that the genericity conditions (GC) and (NR')
hold for the center space $V^c$ of $\der X_h(0)$.

Suppose that $V_0= \bigoplus_{\alpha \in S}\C_\alpha$ with
$S=\bigcup S_i$ as in Theorem~\ref{thm:main-thesis}. Recall that if
any $S_i$ contains a pair of weights $\alpha \ne w \alpha$, where $w$
is the reflection about a Weyl wall, then the $T$-relative equilibria
in the local manifold tangent to $V_0$ are contained in the $G$-orbit
of the manifold tangent to a subspace of $V_0$ that has only one of
the weights $\alpha$ and $w\alpha$. Thus we suppose in the following
that no such pair is contained in $S$. Equivalently, the affine set
$X=\setthat{\xi \in \LT}{V_0\subset \ker \der^2\h-\J^\xi(0)}$ is not
contained in a Weyl wall.

Let $x$ be a point of $V_0$ with momentum $\mu = \J(x)$. Since $G_x
\subset G_\mu$, the investigation of the momenta and their isotropy
plays a key role in the calculation of the isotropy groups.

\begin{lemma}\label{lem:mu_in_LT}
  Let $S$ be a linearly independent set of weights as in
  Theorem~\ref{thm:main-thesis} such that the corresponding affine set
  $X=\setthat{\xi \in \LT}{V_0\subset \ker \der^2\h-\J^\xi(0)}$ is not
  contained in a Weyl wall. Then there is a $G$-invariant splitting
  $\LG^*=\LT^*\oplus \mathfrak c^*$, such that the momenta of the
  points in the local stratum $M$ that bifurcates at $X$ are contained
  in $\LT^*$.
\end{lemma}
\begin{proof}
  Since $X$ is not contained in a Weyl wall, the same holds for the
  set of generators of a point of $M$ which is close to $0$. Thus
  $x\in M$ has a generator $\xi$ not contained in any Weyl wall. Hence
  $G_\xi = T$. Identify $\LG^*$ and $\LG$ via a $G$-invariant inner
  product and consider $\mu :=\J(x)$ as an element of $\LG$. Then
  $[\mu,\xi]=0$, thus $\mu \in \LT$.
\end{proof}
Since the momentum with respect to the $T$-action coincides with the
projection of the momentum to $\LT^*$, we obtain immediately
\begin{corollary}\label{cor:isotropy_momentum_T}
  Let $M$ be as in Lemma~\ref{lem:mu_in_LT}, $x \in M$, $\mu =\J(x)$, and
  $\mu_T$ be the momentum of $x$ with respect to the $T$-action. Then
  $G_\mu=G_{\mu_T}$.
\end{corollary}
\begin{corollary}
  There is a local $T$-equivariant symplectomorphism $\sigma$ from a
  neighborhood of $0$ within $V_0$ to a neighborhood of $0$ within
  $M$, such that $\J(x) = \J(\sigma(x))$.
\end{corollary}
\begin{proof}
  By the equivariant Darboux theorem (\cite[Theorem
  22.2]{guillemin1984symplectic} and \cite{equivariantdarboux}), there
  is a local $T$-equivariant symplectomorphism $\sigma$. Since for
  every $x$ in the domain of $\sigma$ and for every $\xi \in \LT$
  \begin{align*}
    \der(\J^\xi\circ \sigma)(x)
    =\der \J^\xi(\sigma( x)) \circ \der \sigma(x)
    =\omega_{\sigma(x)}(\xi( \sigma(x)), \der \sigma(x) \cdot)\\
    = \omega_{\sigma(x)}(\der \sigma(x) \xi x, \der \sigma(x) \cdot)
    = \omega_{x}(\xi x, \cdot)
    =  \der(\J^\xi)(x)
  \end{align*}
  and $\J^\xi\circ \sigma(0) = 0= \J^\xi(0)$, we obtain
  $\J^\xi\circ \sigma = \J^\xi$.
\end{proof}
Corollary~\ref{cor:isotropy_momentum_T} is the main observation we
need to proceed with the investigation of isotropy types of the
$T$-relative equilibria of $M$  and their
momenta. Actually, we will only apply it to the elements of the space
$V_0$, which form relative equilibria of the linearized Hamiltonian
vector field. For $V_0$, we obtain:
\begin{corollary}\label{cor:H_subset_muT}
  Consider $V_0 $ as in Lemma~\ref{lem:mu_in_LT} and an isotropy
  subgroup $H \subset G$ of the representation $V$. Then for any
  $x \in V_0^H$ with momentum $\mu_T$ with respect to the $T$-action,
  we obtain $H \subset G_{\mu_T}$.
\end{corollary}

\subsection{Isotropy of the generators and  the relative equilibria}

We are now in the position to compare the isotropy groups of the
elements of the local manifolds with those of their tangent spaces at
the origin. So we consider $V_0 = \ker \der^2(\h-\J^\xi)(0)$ with
$V_0 = \bigoplus_{\alpha \in S} \C_\alpha$ and $S$ as in
Lemma~\ref{lem:mu_in_LT}. We will call the open dense set of points of
$V_0$ with minimal isotropy type $\tau$ with respect to the $T$-action
the \emph{main stratum} $(V_0)_\tau$ of $V_0$. It coincides with the
points $x = \sum x_\alpha$ with $0 \ne x_\alpha \in \C_\alpha$ for
every $\alpha \in S$.

We will see that $V_0$ is contained in a larger subspace $\tilde V_0$
of the form $\tilde V_0 =\ker \der^2(\h-\J^\eta)(0)$ with
$\eta \in \LT^K$ for some isotropy subgroup $K \subset G$ of the
adjoint action on $\LG$ such that $K$ contains all isotropy groups of
points the main stratum of $V_0$. Then the following lemma implies
that locally near the origin the map $m_{V_0}$ from $(V_0)_\tau$ to the
stratum of $T$-relative equilibria that bifurcates at $V_0$ preserves
the symmetry of the points of $(V_0)_\tau$.

We fix an inner product on $\LG$ which is invariant with respect to
 the adjoint action. Let $\LT'$ denote the orthogonal complement
 within $\LT$ of the isotropy Lie algebra $\LT_x$ of any point
 $x\in (V_0)_\tau$.

\begin{lemma}\label{lem:isotropy_generator}
  Consider $V_0 = \ker \der^2(\h-\J^\xi)(0)$ for some $\xi \in \LT$
  such that the weights of $V_0$ are linearly independent and
  $X:=\setthat{\zeta \in \LT}{V_0 \subset \ker
    \der^2(\h-\J^\zeta)(0)}$ is not contained in any Weyl
  wall. Suppose that $V_0$ is a subspace of
  $\tilde V_0 = \ker \der^2(\h-\J^\eta)(0)$ with $\eta \in \LT^K$ for
  some isotropy subgroup $K \subset G$ of the adjoint action on $\LG$
  that contains all the isotropy subgroups of the elements of
  $(V_0)_\tau$.

  Then there is a neighborhood $U$ of $0 \in V_0$ such that for every
  $v_0 \in U \cap (V_0)_\tau$ the following holds: Let $L \subset K$
  be the intersection of all isotropy subgroups of elements of $\LT$
  with respect to the adjoint action
  that contain $H:=G_{v_0}$. Then $L$ is an isotropy subgroup of the
  adjoint action and the generator $\xi(v_0)$ of the relative
  equilibrium $m_{V_0}(v_0)$ is contained in $(\LT')^L$.
\end{lemma}
\begin{proof}
  Corollary~\ref{cor:intersection-groups} yields immediately that $L$
  is an isotropy subgroup.
 
  The local map $m_{V_0} \colon (V_0, 0) \to (V,0)$ is given by
  $x \mapsto x+v_1(x, \xi(x))$, where the local map $v_1$ is defined
  in a neighborhood of $(0, \xi)$ within $V_0 \times \LT$ and its image
  is contained in the orthogonal complement $V_1$ of
  $V_0$ with respect to a $G$-invariant inner
  product on $V$. $v_1(x,\zeta)$ is defined as the unique solution to
  \begin{equation*}
    \nabla_{V_1}(\h-\J^\zeta)(x + v_1(x,\zeta))=0,
  \end{equation*}
  see \cite[Section 2.4]{thesis}.  For $x \in V_0$, the generator
  $\xi(x)$ is uniquely defined as an element of $\LT'$, see
  \cite[Section 6.4.1]{thesis}. It is given by the equation
  \begin{align*}
    \nabla_{V_0} g(x, \xi(x)) = 0, \\ \text{with} \quad
    g(x, \zeta) := \left(h- \J^{\zeta}\right)(x + v_1(x, \zeta)).
  \end{align*}
  Now the proof proceeds as follows: In a first step, we show that
  $\nabla_{V_0} g(x, \zeta) \in V_0^H$ if $\zeta \in \LT^L$. Then the
  main idea is to deduce the converse: In a neighborhood of
  $(0, \xi) \in V_0 \times \LT$, $\nabla_{V_0}g(x, \zeta) \in V_0^H$
  implies that $\zeta \in \LT^L$. Since $\nabla_{V_0}g(x, \xi(x)) =0$,
  we then obtain immediately that $\xi(x) \in \LT^L$.

  The proof of the converse relies on the fact that the restriction of
  the map $(x,\zeta) \mapsto \nabla_{V_0}g(x, \zeta)$ to a neighborhood
  of $(0,\xi(0))$ intersected with $(V_0)_\tau \times \LT'$ is a
  diffeomorphism onto an open set of the subbundle of $T(V_0)_\tau$ given by the
  normal spaces to the $T$-orbits. We then simply use a dimension
  argument.
  
  A way to see this diffeomorphism is to consider again the
  decomposition of the map $(x,\zeta) \mapsto \nabla_{V_0}g(x, \zeta)$
  into the maps $\vartheta$ and $\Gamma$ known from equivariant
  transversality theory, as it is done in \cite{thesis} in the proof of
  Theorem~\ref{thm:main-thesis}. This will be done in the second step.
  
  The third step consists of the dimension argument.
  
  \emph{Step 1: $\nabla_{V_0} g(x, \zeta) \in V_0^H$
    for $\zeta \in \LT^L$.}
  
  Let $\mathfrak k$ denote the Lie algebra of $K$.
  There is the corresponding map $\tilde v_1$ from a neighborhood of
  $(0, \eta)$ in $\tilde V_0 \times \mathfrak k$ to the orthogonal
  complement $\tilde V_1$ of $\tilde V_0$ such that
  \begin{equation*}
    \nabla_{\tilde V_1}(\h-\J^\xi)(\tilde x +\tilde v_1(\tilde x,\xi))=0.
  \end{equation*}
  By construction, the map $\tilde v_1$ is $K$-equivariant.

  As argued in
  Remark~\ref{rem:constant_in_l}, $V_0$ is an isotropy subspace of
  $\tilde V_0$ and $\tilde v_1$ can be defined on a
  neighborhood of $(0,\eta)$ that contains $(0,\xi)$ such that the
  restriction of $\tilde v_1$ to (a neighborhood of $(0, \eta)$ in)
  $V_0 \times \LT$ coincides with $v_1$.
  
  Since $L \subset K$, the map
  \begin{equation*}
    \tilde x \mapsto \tilde g (x, \zeta)
    := \left(h- \J^{\zeta}\right)
    ( \tilde x+ \tilde v_1 ( \tilde x, \xi))
  \end{equation*}
  is $L$-invariant if $\zeta \in \LT^L$. Thus the gradient
  $\nabla_{\tilde V_0} \tilde g(\cdot, \zeta)$ is $L$-equivariant if
  $\zeta \in \LT^L$.

 Since $V_0$ is an isotropy subspace of $\tilde V_0$ with
 respect to the $T$-action,
 \begin{equation*}
   \nabla _{\tilde V_0} \tilde g(x, \zeta) =
   \nabla _{V_0} g(x, \zeta)
 \end{equation*}
 if $x \in V_0$. Thus, if
 $x \subset V_0^H$ and $\zeta \in \LT^L$, we have $\nabla_{V_0}
 g(x, \zeta) \in V_0^H$.

 Since $\xi -\eta \in \LT_{v_0}$, $\xi$ and $\eta$ are projected to
 the same element of $\LT'$. By abuse of notation, we will denote this
 projection by $\eta$.

\emph{Step 2: Use the decomposition $\vartheta \circ \Gamma$.}

Let $S=\{\alpha_1, \dots, \alpha_l\}$ be the set of weights of $V_0$
and write $x$ in the form $x= \sum_{i=1}^l x_{\alpha_i}$ with
$x_{\alpha_i} \in \C_{\alpha_i}$.  Now we use ideas from equivariant
transversality theory as in the proof of
Theorem~\ref{thm:main-thesis}. Since $g(\cdot, \zeta)$ is
$T$-equivariant, $g$ is of the form
$g(x,\zeta) = g(p_1, \dots,p_l, \zeta)$ with
$p_i= p_i(x) = \abs{x_{\alpha_i}}^2$. We decompose the map
 \begin{equation*}
   V_0\times \LT' \ni (p_1, \dots, p_l, \zeta)
   \mapsto \nabla_{V_0}g(p_1, \dots, p_l, \zeta)  
 \end{equation*}
 into $\vartheta \circ \Gamma$ with $\vartheta\colon V_0 \times \R^l$
 defined by
 \begin{equation*}
   (x, t) \mapsto \sum_{i=1}^l t_i \nabla p_i(x)
   = 2 \sum_{i=1}^l t_i x_{\alpha_i}
 \end{equation*}
 and $\Gamma: V_0\times \LT' \to V_0 \times \R^l$ given by
 \begin{equation*}
   (x,\zeta) \mapsto (x, \partial_1g(p_1(x), \dots, p_l(x), \zeta),
   \dots, \partial_lg(p_1(x), \dots, p_l(x), \zeta)).
 \end{equation*}
 Here $\partial_i$ denotes the partial derivative of $g$ with respect
 to the coordinate $p_i$. As calculated in \cite[ proof of
   Lemma~6.55, part 1]{thesis}, the derivative
 $\der \Gamma(0, \eta)$ is invertible. Thus $\Gamma$ is a local
 diffeomorphism between a neighborhood $U$ of
 $(0, \xi) \in V_0\times \LT'$ and a neighborhood $O$ of
 $\Gamma(0, \xi) = (0,0) \in V_0\times \R^l$.

 Let $Q^H\subset V_0\times \R^l$ denote the preimage
 $\vartheta^{-1}(V_0^H)$.
 Then we know from above that
 $\Gamma(V_0^H \times (\LT')^L) \subset Q^H$.
 
 Recall that $(V_0)_\tau$ denotes the main stratum of $V_0$ and set
 \begin{equation*}
   Q^H_\tau := Q^H \cap \left( (V_0)_\tau \times \R^l\right).
 \end{equation*}
 We will show in the
 following that actually
 \begin{equation}\label{eq:diffeo_isotropy}
   \Gamma\left(\left((V_0)_\tau^H \times (\LT')^L\right)\cap U\right)
   = Q_\tau^H \cap O
 \end{equation}
 for an admissible
 choice of neighborhoods $U$ and $O$ of $(0, \eta)$ and
 $\Gamma(0,\eta)$ respectively. By definition of $\xi(x)$, we have
 $\vartheta \circ \Gamma(x, \xi(x)) = 0 \in V_0^H$ for
 $x \in V_0$.  Since the map $x \mapsto \xi(x)$ is continuous,
 $(x, \xi(x))$ is contained in $U$ for $x$ small
 enough. Altogether this implies that $\xi(x) \in (\LT')^L$ for
 small $x \in V_0^H$.

 \emph{Step 3: Proof of equation~\ref{eq:diffeo_isotropy}}.

 It remains to prove the claim that we can choose $U$ and $O$ such
 that equation~\ref{eq:diffeo_isotropy} is satisfied.
 
 We note first that for each $x \in V_0$ the map $\vartheta(x, \cdot)$
 is a linear map from $\{x\} \times \R^l$ to $V_0$. Thus
 $Q^H\cap \left(\{x\} \times \R^l\right)$ is a linear subspace of
 $\{x\} \times \R^l$. Moreover, by Corollary~\ref{cor:H_subset_muT}
 the space $V_0^H$ is contained in the set of points $x$ of $V_0$ with
 $\J(x) = \J_T(x) \in \LT^L$. Since
 $\J(x) = \pi\sum_{i=1}^l \abs{x_{\alpha_i}}^2 \alpha_i$, we obtain
 that
 \begin{equation*}
   \J \circ \vartheta (x, t)
   = 4\pi \sum_{i=1}^l t_i^2\abs{x_{\alpha_i}}^2 \alpha_i.
 \end{equation*}
 Thus for $x \in (V_0)^H
 _\tau$, the map $\J \circ \vartheta(x,\cdot)$ is
 a diffeomorphism from $\{x\} \times \R_+^l$ onto its image in $\LT'$.
 By the $G$-equivariance of $\J$, we have $\J\circ \vartheta (Q^H)
 \subset (\LT')^L$. Hence $Q_H \cap \left(\{x\} \times \R_+^l\right)$ is contained
 in the submanifold $\J^{-1}((\LT')^L) \cap  \left(\{x\} \times \R_+^l\right)$, which has the
 same dimension as $(\LT')^L$. Since $\Gamma$ is a diffeomorphism and
 \begin{equation*}
   \Gamma( \{x\} \times (\LT')^L) \subset Q_H \cap \left(\{x\} \times \R^l\right),
 \end{equation*}
 we obtain that the spaces $Q_H \cap \left( \{x\} \times \R^l\right)$
 and $(\LT')^L$ are of the same dimension.

 Since $\Gamma$ is a local diffeomorphism, its restriction to the set
 $V_0^H \times (\LT')^L$ is also a local diffeomorphism onto its
 image. As shown above, its image is contained in the space
 $\left(V_0^H \times \R^l\right) \cap Q^H$. 

 $\Gamma$ is of the form
 \begin{equation*}
   (x,\zeta) \mapsto (x, \gamma(x, \zeta))
 \end{equation*}
 with
 \begin{equation*}
   \gamma(x,\zeta)
 := (\partial_1 g(p_1(x), \dots, p_l(x),\zeta), \dots, \partial_l
 g(p_1(x), \dots, p_l(x),\zeta).
 \end{equation*}
 Hence for $x \in V_0^H$, the map $\zeta \mapsto \Gamma(x, \zeta)$ is
 an immersion from $(\LT')^L$ to the space
 $Q^H \cap \left(\{x\} \times \R^l\right)$, which has the same
 dimension as $(\LT')^L$. Thus the map is
 locally surjective.

 Now choose $U$ and $O=\Gamma(U)$ such that for each $x \in V_0^H$ the
 intersection $O \cap \left(\{x\} \times \R^l\right) \cap Q^H$ is
 connected. (Since $\left(\{x\} \times \R^l\right)\cap Q^H$ is a
 linear subspace of $\{x\} \times \R^l$, we may for example take a
 ball for $O$.) Since $\Gamma$ is a diffeomorphism from
 $\left(\{v_0\} \times (\LT')^L\right)\cap U$ onto its image in a
 space of the same dimension, its image is an open subset of
 $Q^H\cap \left(\{v_0\}\times \R^l\right) \cap O$. Moreover,
 $\left(\{v_0\} \times (\LT')^L\right)\cap U$ is closed in $U$. Thus
 its image is closed in $O$ and hence closed in
 $Q^H\cap \left(\{v_0\}\times \R^l\right) \cap O$. By connectedness,
 \begin{equation*}
   \Gamma\left( \left(\{v_0\} \times (\LT')^L\right)\cap U \right)
   =  Q^H\cap \left(\{v_0\}\times \R^l\right) \cap O.
 \end{equation*}
\end{proof}

\begin{corollary}\label{cor:mV0equiv}
  Consider $V_0$ and $\tilde V_0$ as in
  Lemma~\ref{lem:isotropy_generator} that contains $V_0$.  Then for
  every $V_0$ locally
  \begin{equation*}
    m_{V_0} \colon (V_0, 0) \to (V,0)
  \end{equation*}
 satisfies
  $K_{v_0}\subset K_{m_{V_0}(v_0)}$ for $v_0 \in V_0$. 
\end{corollary}
\begin{proof}
 First of all, we note that every point $v_0$ of $V_0$ is contained in the
 main stratum of some isotropy subspace of $V_0$ with respect to the
 $T$-action. Replacing $V_0$ with this subspace if
 necessary, we can assume that $v_0$ is contained $(V_0)_\tau$. Since
 the number of isotropy subspaces of $V_0$ is finite, we can find
 $\eps>0$ such that Lemma~\ref{lem:isotropy_generator} applies for all
 $v_0 \in B_\eps(0) \subset V_0$. Then $\xi(v_0) \in \LT^L$ with $L$
 as in Lemma~\ref{lem:isotropy_generator}. Hence
 $v_1(\cdot, \xi(v_0))$ is $L$-equivariant and thus
 \begin{equation*}
   m(v_0) = v_0 + v_1(v_0, \xi(v_0)) \in V^H.
 \end{equation*}
\end{proof}

\begin{remark}\label{rem:characterization_L_part1}
  The proof of Lemma~\ref{lem:isotropy_generator} also implies that
  for $x \in (V_0)_\tau^H$ and $\LT^H = \LT^L$ the sets
  $\left((\J\circ \vartheta)^{-1}( (\LT')^L)\right) \cap \left(\{x\} \times
    \R_+^l\right)$ and $Q^H\cap \left(\{x\} \times \R_+^l\right)$
  coincide. Hence the preimage of $(\LT')^L$ under
  $\J \circ \vartheta(x, \cdot) \colon \R_+^l \to \LT'$ coincides with
  the intersection of $\R_+^l$ with a linear subspace of $\R^l$.
 
The map $\J \circ \vartheta (x, \cdot)$ is given by
\begin{equation*}
  t \mapsto \sum_{i=1}^l t_i^2\abs{x_{\alpha_i}}^2 \alpha.
\end{equation*}
Thus for $x \in (V_0)_\tau$, it is a composition of a linear
isomorphism $\R^l \to \LT'$ and the map
\begin{equation*}
  q:(t_1, \dots, t_l) \mapsto   (t_1^2, \dots, t_l^2).
\end{equation*}
Thus the linear subspace $A \subset \R^l$ given by the vectors $t$
with $(x, t) \in Q^H\cap \left(\{x\} \times \R^l\right)$ has the
property that $q(A)$ coincides with the intersection of $\R_+^l$ with
a linear subspace of $\R^l$ of the same dimension. Suppose that the
dimension of $A$ is $k$. After reordering of the coordinates
$t_1, \dots t_l$ if necessary, we can choose a basis $x_1, \dots, x_k$
of $A$ of $k$ vectors which form an $l\times k$-matrix $B$ in column
echelon form:
\begin{equation*}
  B =
  \begin{pmatrix}
    1& 0& \cdots &0\\
    0&1& \cdots &0\\
    0&0&\ddots&0\\
    0&0& \dots &1\\
    t_{k+1,1}& t_{k+1,2} &\cdots &t_{k+1,k}\\
    \vdots&\vdots&\vdots&\vdots\\
    t_{l,1}& t_{l,2} &\cdots &t_{l,k}
  \end{pmatrix}.
\end{equation*}
Then $q(A)$ is given by the intersection of $\R_+^l$ and the span of
$q(x_1), \dots, q(x_k)$.

Thus for any $1 \le i<j\le k$ and the vector $q(x_i +x_j)$ is equal to
$q(x_i) + q(x_j)$ and thus the products $t_{s,i}t_{s,j}$ vanish for
$s=k+1, \dots, l$. Hence in each row of $B$, at most one entry is
different from $0$. Therefore also the space that contains $q(A)$ as
an open subset and also the the space $(\LT')^L$ has this
property: $(\LT')^L$ has a basis $e_1, \dots e_j$ such that there is a
partition of $S$ into disjoint subsets $S_{e_1}, \dots, S_{e_k}$ and
$e_j$ is contained in the affine span of the elements of $S_{e_j}$.
\end{remark}

To show that $G_{v_0} \subset G_{m_{V_0}(v_0)}$ is locally true in general,
we therefore have to construct $\tilde V_0$ such that $G_{v_0} \subset
K$. To find an appropriate subgroup $K$ with Lie algebra $\mathfrak k$,
we introduce the notion of \emph{orthogonal intersection}, which
slightly differs from orthogonality .
\begin{definition}
  Two subspaces $U_1$ and $U_2$ of an inner product space
  \emph{intersect orthogonally} if their orthogonal projections to the
  orthogonal complement of $U_1 \cap U_2$ are orthogonal to each
  other.

  Two affine subspaces $A_1$ and $A_2$ \emph{intersect orthogonally}
  if their underlying subspaces do and in addition the intersection
  $A_1 \cap A_2$ is non-empty.
\end{definition}
\begin{lemma}\label{lem:orthogonal_intersection}
  Let $V_0$ be as in Lemma~\ref{lem:mu_in_LT}.
  Consider $x \in (V_0)_\tau$. Let $L$ denote the intersection of all
  coadjoint isotropy subgroups $M$ of elements of $\LT^*$ with $G_x
  \subset M$. Then $L$ is a coadjoint isotropy subgroup of an element
  of $\LT^*$, and $\left(\LT^*\right)^L$ intersects the affine span of
  $S$ orthogonally.

  If $S_i \subset S$ is a subset of $S$ that corresponds to weight
  spaces of a particular eigenspace of $\der^2 \h(0)$, then
  $\left(\LT^*\right)^L$ even intersects the affine span of
  $S_i$ orthogonally.
\end{lemma}
\begin{proof}
  Again by Corollary~\ref{cor:intersection-groups}, $L$ is a coadjoint
  isotropy subgroup of an element of $\LT^*$.

  Since $G_x \subset L$, we have $G_x = L_x$.  We split $V^c$ into
  $L$-invariant components of the following form: First of all,
  consider the eigenspaces $U_i$ of $\der^2 \h(0)|_{V^c}$, which are
  irreducible complex $G$-representations by our genericity assumption. Now we
  consider the orthogonal projection to $\left(\LT^*\right)^L$ and
  split each $U_i$ further into subspaces $U_i^j$,
  $j =1, \dots, l_i$ given by the sums of weight spaces corresponding
  to the weights of $U_i$ that project to the same element of
  $\left(\LT^*\right)^L$. By Lemma~\ref{lem:orthogonal_affine} the
  spaces $U_i^j$ are $L$-invariant.

  Let $x^j_i \in U_i^j$ denote the corresponding components of $x \in
  (V_0)_\tau$. Then $L_x = \bigcap_{i,j} L_{x^j_i}$.
  
  Set $\mu^j_i := \J(x^j_i)$. Recall that $L_{x^j_i}\subset G_{x^j_i}
  \subset G_{\mu^j_i}$. From $G_x \subset G_{\mu^j_i}$ together with
  $\mu^j_i \in \LT^*$ and the definition
  of $L$, we obtain that
  $L \subset G_{\mu^j_i}$ for every $\mu^j_i$.

  Let $S^j_i \subset S$ denote the subset of weights corresponding to
  weight spaces contained in $U_i^j$. Let $x_\alpha$ denote the
  $\C_\alpha$-component of $x$ with respect to the isomorphism $V_0
  \simeq \bigoplus_{\alpha \in S} \C_\alpha$. Then
  \begin{equation*}
    \mu^j_i = \pi \sum_{\alpha \in S^j_i}\abs{x_\alpha}^2 \alpha.
  \end{equation*}
  If we assume w.~l.~o.~g. that
  $\sum_{\alpha \in S^j_i}\abs{x_\alpha}^2 = 1$ (otherwise we consider
  a multiple of $\mu^j_i$), we obtain that
  $\mu^j_i \in \left(\LT^*\right)^L$ is contained in the affine span
  of $S^j_i$. Then by definition of $S^j_i$, this implies that
  $\left(\LT^*\right)^L$ intersects the affine span of each $S^i_j$
  orthogonally. Thus this also holds for the affine spans of the sets
  $S_i = \bigcup_{j=1}^{l_i}S^j_i$ and $S = \bigcup_i S_i$.
\end{proof}

Given the affine span $A$ of the set $S$ of weights of $V_0$, we
consider the intersections of Weyl walls that intersect $A$
orthogonally. Note that the set of these intersection of Weyl walls is
closed under intersections: Let $U$ denote the underlying subspace of
$A$. Then $A = a + U$ for some $a \in U^\perp$. A subspace $I$
intersects $A$ orthogonally iff its orthogonal projection to $U$
coincides with $I \cap U$ and $a \in I$. If two subspaces $I$ and $J$
intersect $A$ orthogonally, then $a \in I \cap J$ and the orthogonal
projection of $I \cap J$ to $U$ coincides with $I \cap J \cap U$.

We identify the empty intersection of Weyl walls with $\LT^*$, which
obviously intersects $A$ orthogonally.
Thus we can find a minimal element $I_{\min}$ of all intersections of
Weyl walls that are orthogonal to $A$ with respect to the partial
ordering $I \le J$ iff $I \subset J$. Then we choose $K$ to be the
isotropy subgroup of the coadjoint action such that
$I_{\min} = (\LT^*)^K$.

\begin{lemma}\label{lem:tildeV01}
  Suppose that $V^c$ is an irreducible $G$-symplectic representation
  and that $\h:V\to \R$ is a smooth $G$-invariant function with
  $\der h(0) = 0$ that satisfies the genericity assumptions (GC) and
  (NR'). Consider
  $\ker \der^2(\h-\J^\xi)(0) = V_0 = \bigoplus_{\alpha \in S}
  \C_\alpha$ such that the affine span $A$ of $S$ does not contain the
  origin. Let $I_{\min}$ be the minimal element within all
  intersections of Weyl walls that intersect $A$ orthogonally.  Then
  for some isotropy group $K \subset G$ of the coadjoint action,
  $I_{\min}$ coincides with the isotropy subspace $(\LT^*)^{K}$. For
  every $x \in (V_0)_\tau$, the isotropy group $G_x$ is contained in
  $K$. Moreover, there is an $\eta \in \LT^K$ with
  \begin{equation*}
    V_0 \subset \ker \der^2(\h-\J^\eta)(0) =: \tilde V_0.
  \end{equation*}
\end{lemma}
\begin{proof}
  Since $I_{\min}$ is an intersection of Weyl walls, by
  Lemma~\ref{lem:BTD} there is an isotropy group $K \subset G$ of the
  coadjoint action such that this intersection coincides with
  $(\LT^*)^K$. Let $I_{\min}^\perp$ be the orthogonal subspace of
  $I_{\min}$ within $\LT^*$. Extend the set $S$ to the set $\tilde S$
  given by the weights of $V$ contained in
  \begin{equation*}
    A + I_{\min}^\perp= (A \cap I_{\min}) + I_{\min}^\perp.
  \end{equation*}
  Let $\tilde V_0$ denote the sum $\bigoplus_{\alpha \in \tilde S}
  \C_\alpha$. Obviously, $V_0 \subset \tilde V_0$.
  
  Consider $c \in \R$ with $\der^2\h(0) =2 \pi c\id$ on $V^c$.

  Then  $\xi$ satisfies ($\langle \alpha, \xi \rangle = c $ iff
  $\alpha \in A$). Similarly,
  $\tilde V_0 = \ker \der^2(\h-\J^{\eta}(0))$ is equivalent to
  ($\langle \alpha, \eta\rangle = c $ iff $\alpha \in  A +
  I_{\min}^\perp$).

  Fixing a $G$-invariant inner product on $\LG$, we identify $\LG$
  with $\LG^*$ and $\LT$ with $\LT^*$.  Let $\eta$ denote the
  orthogonal projection of $\xi$ to $I_{\min}$.  Since
  $\eta - \xi \in I_{\min}^\perp$, the inner products of $\eta$ and
  $\xi$ with elements of $I_{\min}$ coincide. Thus
  $\langle \alpha, \eta\rangle = c$ if $\alpha \in A \cap I_{\min}$
  and $\langle \alpha, \eta\rangle \ne c$ if
  $\alpha \in I_{\min}\setminus (A \cap I_{\min})$. Moreover
  $\eta \in I_{\min}$ implies that for any weight $\alpha$ the value
  $\langle \alpha, \eta\rangle $ coincides with the orthogonal
  projection of $\alpha$ to $I_{\min}$ evaluated at $\eta$. Therefore
  $\langle \alpha, \eta\rangle = c$ iff
  $\alpha \in (A\cap I_{\min}) + I_{\min}^\perp = A+
  I_{\min}^\perp$. By the maximality of $\tilde S$ within
  $A + I_{\min}^\perp$, we obtain
  $\tilde V_0 = \ker \der^2(\h-\J^{\eta}(0))$.
  
  Next, we have to show that all isotropy subgroups of elements
  $x = \sum_{\alpha \in S} x_\alpha$ with
  $0 \ne x_\alpha \in \C_\alpha$ of $V_0$ are contained in $K$.  By
  Lemma~\ref{lem:orthogonal_intersection}, there is an intersection of
  Weyl walls $\left(\LT^*\right)^L$ that intersects $S$ orthogonally
  such that $G_x \subset L$. By definition of $K$, we have
  \begin{equation*}
    \left(\LT^*\right)^K = I_{\min} \subset \left(\LT^*\right)^L
  \end{equation*}
  and hence $G_x \subset L \subset K$.
\end{proof}

Now, we extend the result to the case that $V$ is not necessarily
irreducible.
\begin{corollary}\label{cor:existence_tilde_V_0}
  Suppose that $\h:V\to \R$ is a smooth
  $G$-invariant function with $\der h(0) =
  0$ that satisfies the genericity assumptions (GC) and (NR') and
  consider $\ker \der^2(\h-\J^\xi)(0) = V_0 = \bigoplus_{\alpha \in S}
  \C_\alpha \subset V^c\subset V$ for some $\xi \in
  \LT$. Then there is an isotropy group $K \subset
  G$ of an element of
  $\LT^*$ with respect to the coadjoint action such that for any $x
  \in (V_0)_\tau$ the isotropy group $G_x$ is contained in
  $K$ and such that $V_0$ is contained in a
  $K$-invariant subspace $\tilde V_0$ of the form $\tilde V_0 = \ker
  \der^2 (\h-\J^\eta) (0)$ for some $\eta \in \LT^{K}$.
\end{corollary}
\begin{proof}
  We split $V^c$ into eigenspaces $U_i$ of $\der^2 \h(0)$,
  which are complex irreducible $G$-representations by our genericity
  assumption. In addition, we split $V_0$ into subspaces $V_0^i$
  contained in these eigenspaces. Let $S_i$ denote the set of weights
  of $V_0^i$, and let $A_i$ be the affine span of $S_i$.

  Then the set of Weyl wall intersections that intersect each $A_i$
  orthogonally contains a minimal element, which coincides with
  $(\LT^*)^K$ for some isotropy subgroup $K\subset G$ of the coadjoint
  action. Lemma~\ref{lem:orthogonal_intersection} implies that
  $G_x \subset K$ for every $x \in (V_0)_\tau$.

  Again, let $\eta$ denote the orthogonal projection of $\xi$ to
  $\LT^K$.  As above, we obtain that $V_0^i\subset \ker \der^2(\h -
  \J^\eta)(0) =: \tilde V_0$ for every $i$ and hence $V_0 \subset
  \tilde V_0$.
\end{proof}
\begin{remark}\label{rem:characterization_L_part2}
  Recall that by Remark~\ref{rem:characterization_L_part1}, if
  $(\LT')^L=(\LT')^H$ for some isotropy subgroup $H \subset G$ of
  $(V_0)_\tau$ and a coadjoint isotropy subgroup $L \subset G$, then
  $(\LT')^L$ has a basis $e_{S_1}, \dots, e_{S_k}$ corresponding to a
  partition $S = \dot\bigcup_{j=1}^k S_j$ such that each $e_{S_j}$ is
  contained in the affine span of $S_j$.

  Now, we have the additional condition of orthogonal intersection of
  $\LT^L$ with the affine span of $S$, and we have to investigate how
  the two conditions fit together. Suppose that $x\in (V_0)_\tau$
  satisfies $\J(x) \in \LT^L$. As argued in the proof of
  Lemma~\ref{lem:orthogonal_intersection}, we can split $S$ into sets
  $\tilde S_i$ of weights whose orthogonal projection to $\LT^L$
  coincides. Then the spaces
  $\bigoplus_{\alpha \in \tilde S_i} \C_\alpha$ are $L$-invariant by
  Lemma~\ref{lem:orthogonal_affine}. If $x_{\alpha}$ denotes the
  $\C_\alpha$-component of $x$ for $\alpha \in S$ and
  $x_i := \sum_{\alpha \in \tilde S_i} x_\alpha$, then
  $G_x = \bigcap_i G_{x_i}$. Thus if $L$ is the smallest coadjoint
  isotropy subgroup that contains $H:= G_x$ (equivalently
  $\LT^L = \LT^H$), then we necessarily have
  $\mu_i:=\J(x_i) \in \LT^L$ for every $i$. Since each set
  $\tilde S_i$ contains a positive multiple of $\mu_i$ in its affine
  span and since the $\mu_i$ form a linearly independent set, the
  partition $S=\dot\bigcup_j S_j$ coincides with the partition
  $S = \dot \bigcup_i \tilde S_i$. Hence $\LT^L$ is orthogonal to the
  affine spans of the sets $S_i$.  Moreover, since the intersection
  point of $\LT^L$ and the affine span of $S_j=\tilde S_j$ is a
  positive multiple of $\mu_j$, it is contained in the inside of the
  convex hull of the set $S_j$ if $S_j$ contains more than one
  element.

  Altogether, to compute the isotropy subgroups of the relative
  equilibria that we obtain from Theorem~\ref{thm:main-thesis}, it suffices to
  investigate the case that a coadjoint isotropy subspace intersects
  the convex hull of a linearly independent set $S$ of weights in a single
  inner point $\nu$ and to compute the isotropy subgroups of the
  points $x\in \bigoplus_{\alpha \in S}\C_\alpha$ with $\J(x)= \nu$. 
\end{remark}

\begin{theorem}
  Let $V$ be a $G$-symplectic representation and $\h: V \to \R$ be a
  smooth $G$-invariant Hamiltonian function with $\der h (0) = 0$ such
  that (GC) and (NR') hold for $h$. If $\xi \in \LT$ and
  $V_0 := \der^2 (\h- \J^\xi)(0)$ has linearly independent weights,
  then there is a local map $m_{V_0}$ from a neighborhood $U$ of $0$
  in $V_0$ to $V$ such that $M_0 :=m_{V_0}(U)$ is a manifold of
  $T$-relative equilibria with $0 \in M_0$, $T_0 M_0 = V_0$ and
  $G_{m_{V_0}(x)} = G_x$ for every $x \in V_0$.
\end{theorem}
\begin{proof}
  We only have to prove that $G_{m_{V_0}(x)} = G_x$ for every
  $x \in V_0$, if $U$ is small enough. $G_x \subset G_{m_{V_0}(x)}$
  follows from Corollary~\ref{cor:existence_tilde_V_0} and
  Corollary~\ref{cor:mV0equiv}.

  Thus it remains to prove that we can find $U$ with $G_{m_{V_0}(x)}
  \subset G_x$ for every $x \in U$.
  
  We recall that $m_{V_0}(x) = x + v_1(x, \xi(x))$ and that
  $v_1$ satisfies $(\der_x v_1)(0, \xi(0))= 0$ and
  $(\der_\xi v_1)(0, \xi(0)) =0$. Hence the derivative of the map
  $\phi\colon x \mapsto v_1(x, \xi(x))$ vanishes as well at $x =0$.

  Let $H \subset G$ be an isotropy subgroup of $V$. Split $V_0$ into
  the fixed point space $V_0^H$ and its orthogonal complement $C_0$
  within $V_0$. (The case $C_0 = \{0\}$ is trivial. Thus we assume $C_0
  \ne \{0\}$ in the following.) If $S_{C_0}$ denotes the
  unit sphere of $C_0$, then
  \begin{equation*}
    d := \dist(S_{C_0}, V^H) > 0
  \end{equation*}
  and we obtain for every $c_0 \in C_0$ that $\dist(c_0, V^H) \ge d
  \norm{c_0}$.

  Now we suppose that for some $x = v_0^H + c_0 \in V_0$ with
  $v_0^H \in V_0^H$ and $c_0 \in C_0$ we have $m_{V_0}(x) \in V^H$. It
  is
  \begin{equation*}
    m_{V_0}(v_0^H+c_0) = v_0 + c_0 + \phi(v_0^H + c_0) = v_0^H + c_0 +
    \phi(v_0^H) + \der_{V_0}\phi(v_0^H)c_0 + R(v_0^H, c_0)
  \end{equation*}
  with $\lim_{c_0 \to 0} \frac {R(v_0^H, c_0)} {\norm{c_0}} = 0$.

  Since $v_0^H + \phi(v_0^H)= m_{V_0}(v_0^H) \in V^H$, we deduce
  \begin{equation*}
    c_0 + \der_{V_0}\phi(v_0^H)c_0 + R(v_0^H, c_0) \in V^H.
  \end{equation*}
  Since $\der_{V_0}\phi(0) = 0$ and $\phi$ is continuously
  differentiable, there is $\delta_1$ such that
  \begin{equation*}
    \norm{\der_{V_0}\phi(v_0^H)} \le \frac d 3
  \end{equation*}
  if $\norm{v_0^H} < \delta_1$.
  
  By Hadamard's lemma, there is a smooth function $f$ with values in
  $L(C_0, C_0)$ such that
  $R(v_0^H, c_0)= \langle c_0, f(v_0^H, c_0) c_0 \rangle$. Since $f$
  is bounded on $B_{\delta_1}(0)$, we can find $\delta_2 \le \delta_1$
  such that $R(v_0^H, c_0) \le \frac d 3 \norm{c_0}$ if
  $\norm{v_0^H} < \delta_1$ and $\norm{c_0} < \delta_2$. Thus if
  
   $\norm{x} < \min\left(\delta_1, \delta_2\right)$, then
   \begin{equation*}
     \norm{\der_{V_0}\phi(v_0)c_0 + R(v_0^H, c_0)}\le \frac 2 3 d \norm{c_0}.
   \end{equation*}
   Therefore $\der_{V_0}\phi(v_0)c_0 + R(v_0^H, c_0) \in V^H$ implies
   that $c_0 = 0$ and hence $x \in V_0^H$.
\end{proof}

\section{Dimensions of the isotropy strata}\label{sec:dimensions}

As shown by Patrick and Roberts \cite{Patrick2000TheTR}, for free
Hamiltonian actions of compact connected groups $G$ the relative
equilibria generically form a Whitney stratified set, whose strata are
determined by the conjugacy classes of the groups $K=G_\xi \cap G_\mu$
consisting of the common elements of the isotropy subgroups of the
generator $\xi$ and the momentum $\mu$. For a class $(K)$ the stratum
has the dimension
\begin{equation*}
  \dim G + 2 \dim Z(K) - \dim K,
\end{equation*}
where $Z(K)$ denotes the center of $K$.  This can be generalized to
possibly non-free actions, see \cite[Section 6.3]{thesis}. Generically
Patrick and Robert's results apply to the subspaces of elements with
the same isotropy group $H$ and the free action of
$\left(\quot{N(H)}{H}\right)^\circ$ on that subspace. The strata are
determined by the isotropy types $(H)$ of the relative equilibria and
the conjugacy classes of
\begin{equation*}
  K:=\left(\quot{N(H)}{H}\right)^\circ_\xi
  \cap \left(\quot{N(H)}{H}\right)^\circ_\mu,
\end{equation*}
where $\xi$ denotes the generator within the Lie algebra
$\quot {\mathfrak n(\mathfrak h)}{\mathfrak h}$ and $\mu$ the momentum
in $\left(\quot {\mathfrak n(\mathfrak h)}{\mathfrak h}\right)^*$. If
$H = gH'g^{-1}$ we identify the groups
$\left(\quot{N(H)}{H}\right)^\circ$ and $\left(\quot{N(H')}{H'}\right)^\circ$
via $n \mapsto gng^{-1}$ and accordingly we identify conjugacy classes
of their subgroups. Then the stratum corresponding to the types $(H)$
and $(K)$ for $K \subset \quot{N(H)}{H}$ has the dimension
\begin{equation*}
  \dim G - \dim H + 2 \dim Z(K) - \dim K.
\end{equation*}
As shown above, the isotropy groups of the local manifold $M_0$ of
$T$-relative equilibria tangent to $V_0$ coincide with those of $V_0$
if $V_0=\bigoplus_{\alpha \in S} \C_\alpha$ with $S=\bigcup S_i$ as in
Theorem~\ref{thm:main-thesis} such that no $S_i$ contains a
Weyl reflection pair $\alpha \ne w \alpha$.
To calculate the generic dimensions of the strata that intersect
$M_0$, we need to investigate the group
$\left(\quot{N(H)}{H}\right)^\circ$ and the subgroup $K$.

The group $\left(\quot{N(H)}{H}\right)^\circ$ is isomorphic to
$\quot{\left(C_G(H)\right)^\circ}{\left(Z(H)\right)^\circ}$, where
$C_G(H)$ denotes the centralizer of $H$ within $G$, see for instance
\cite[Corollary~3.10.1]{field2007dynamics}. Equivalently, this
holds for the Lie algebras:
\begin{equation*}
  \quot{\mathfrak n(H)}{\mathfrak h}
  \cong \quot{\mathfrak c_G(H)}{\mathfrak z(H)}.
\end{equation*}
\begin{remark}
  In contrast to Field \cite{field2007dynamics}, we denote the Lie
  algebra of $C_G(H)$ by $\mathfrak c_G(H)$ and not by
  $\mathfrak c_G(\mathfrak h)$ since it depends on $H$ and not only on
  $\mathfrak h$: Consider for example a finite subgroup which is not
  contained in the center. Consequently, the Lie algebra
  $\mathfrak n(H)$ depends on $H$. Similarly, $\mathfrak z(H)$ depends
  on the group $H$ in general: Suppose for example that $H = N(T)$ and
  that the Weyl group $\quot {N(T)}{T}$ is non-trivial.

  Note also that $\left(\quot {N(H)}{H}\right) ^\circ$ and
  $\left(\quot {N(H^\circ)}{H^\circ}\right) ^\circ$ do not coincide in general.
\end{remark}
Since $\LT_x$ coincides for all $x \in (V_0)_\tau$ and contains a point
$\xi \in \LT$ which is not contained in any Weyl wall, we obtain
\begin{equation*}
  \mathfrak c_G(H) \subset \mathfrak c_G(\xi) = \LG_\xi = \LT.
\end{equation*}
Thus the Lie algebra $\quot{\mathfrak n(H)}{\mathfrak h}$ is
Abelian and hence
\begin{equation*}
  K = Z(K) = \left(\quot{N(H)}{H}\right)^\circ.
\end{equation*}
Consequently, the dimension of the stratum only depends on the
isotropy type $(H)$.

We need to determine the dimension of the group
$\left(\quot{N(H)}{H}\right)^\circ$. Again let $L$ denote the
minimal isotropy group of the adjoint action within $\LT$ that contains
$H$.
From
$\mathfrak c_G(H) \subset \LT$, we obtain
\begin{equation*}
  \mathfrak c_G(H) = \mathfrak c_G(H) \cap \LT
  = \LT^{H} = \LT^L. 
\end{equation*}
Moreover,
\begin{equation*}
  \mathfrak z(H)
  = \mathfrak h \cap \mathfrak c_G(H)
  =  \mathfrak h \cap \LT^L= \LT_x \cap \LT^L
\end{equation*}
for any $x \in \left(V_0\right)_\tau$. As above, let $\LT'$ denote the
orthogonal complement of $ \LT_x$ within $\LT$. Note that we can
identify $\LT'$ with the span of the set of weights $S$ of
$V_0$. Recall that $\left(\LT^*\right)^L$ intersects the affine span
of $S$ orthogonally, see
Lemma~\ref{lem:orthogonal_intersection}. Hence $\LT^L$ intersects
$\LT'$ orthogonally, and thus
$\LT^L=\left(\LT_x \cap \LT^L\right) \oplus \left(\LT' \cap
  \LT^L\right)$.  Therefore
\begin{equation*}
  \quot{\mathfrak n(H)}{\mathfrak h} \cong \left(\LT'\right)^L.
\end{equation*}
Altogether, the dimension of the stratum that contains $x \in V_0^H$
is given by
\begin{equation*}
  \dim G -\dim H + \dim  \left(\LT'\right)^L .
\end{equation*}

\begin{remark}\label{rem:lie-algebra-strata}
  Despite the fact that $\left(\quot {N(H)}{H}\right) ^\circ$ and
  $\left(\quot {N(H^\circ)}{H^\circ}\right) ^\circ$ can have different
  dimensions, it is often sufficient to compute the Lie algebras in
  order to calculate the dimensions of the strata: For example, we know
  that $G_x \subset G_\mu$. If $\LG_\mu$ is the minimal isotropy Lie
  algebra that contains $\LG_x$, we can conclude that $L = G_\mu$.

  Moreover, by Remarks~\ref{rem:characterization_L_part1} and
  \ref{rem:characterization_L_part2}, we obtain further restrictions
  on $L$: We have seen that for points of $(V_0)_\tau$ we only have to
  consider coadjoint isotropy subgroups $M$ such that the set of
  weights of $V_0$ splits into subsets each of which has an affine
  span that intersects $(\LT^*)^M$ orthogonally in a single point
  contained in its convex hull. Thus for a given point
  $x \in (V_0)_\tau$ with momentum $\mu$ we determine the minimal
  coadjoint isotropy group $M$ of this type such that
  $M \subset G_\mu$. Then we know that $G_x \subset M$. Thus if the
  Lie algebra $\mathfrak m$ is the minimal Lie algebra of coadjoint
  subgroups of this type that contains $\mathfrak g_x$ , we deduce
  that $L = M$.
\end{remark}

\section{Examples of non-constant isotropy within the main stratum}

The $T$-relative equilibria that exist by
Theorem~\ref{thm:main-thesis} form a Whitney-stratified local set and
the strata have constant isotropy type with respect to the $T$-action.
This leads to the question if the $T$-relative equilibria of such a
stratum have the same isotropy type with respect to the $G$-action. In
general, this is not the case. In this section, we will discuss some
examples of this phenomenon.

Therefore we do not expect in general that the $G$-orbit a stratum of
$T$-relative equilibria consists of one stratum within the
Whitney-stratified local set of $G$-relative equilibria. Nevertheless,
this set is stratified by the $G$-isotropy type as we have shown in the
last section.

Since for every kernel $V_0 = \ker \der^2 (\h-\J^\xi)(0)$ as above the
local map $m_{V_0}$ preserves the $G$-isotropy groups of the points
contained in the principal stratum $(V_0)_\tau$ with respect to the
$T$-action, the computation of the isotropy groups or the isotropy Lie
algebras is pure representation theory. However, we will see that some
of our considerations that come from Hamiltonian dynamics are still
helpful:

It suffices to consider
$V_0\simeq\bigoplus_{\alpha \in S} \C_\alpha$ such that $S$ is as in
Theorem~\ref{thm:main-thesis} and each $S_i$ contains no pair
$\alpha \ne w\alpha$ for any reflection $w$ about a Weyl wall.

First of all, recall that $G_x \subset G_\mu$ if $\mu =: \J(x)$. Thus,
we only have to investigate the $G_\mu$-action. Moreover, we know from
Lemma~\ref{lem:mu_in_LT} that $\mu$ coincides with its projection to
$\LT^*$.  Therefore,
\begin{equation*}
  \mu =\pi \sum_{\alpha \in S}|x_{\alpha}|^2 \alpha,
\end{equation*}
when $x = \sum_{\alpha \in S} x_{\alpha}$ with
$x_{\alpha} \in \C_{\alpha}$.

W.~l.~o.~g. we assume that $x_{\alpha} \ne 0$ for every $\alpha \in
S$. Then by Corollary~\ref{cor:existence_tilde_V_0} $G_x$ is even
contained in the maximal subgroup $K \subset G_\mu$ such that $\LT^K$
intersects the affine span of $S$ orthogonally. To compute $G_x$, it
suffices to consider the components of $x$ with respect to a splitting
of $V$ into $K$-irreducible subspaces. Then $G_x=K_x$ is given by the
intersection of the isotropy subgroups of these components.

\begin{remark}\label{rem:irreducible-components}
  Recall that by 
  Remark~\ref{rem:characterization_L_part2} it suffices to consider
  the case that $\LT^K$ intersects the affine span
  of the set $S$ of weights of $V_0$ orthogonally in exactly one point
  contained in the inside of the convex hull of $S$.
  
  Moreover, we will assume w.~l.~o.~g.
  in the following that $V_0$ is contained in an irreducible
  $K$-subrepresentation, where $K$ is the maximal isotropy subgroup of
  the adjoint action such that $\LT^K$ intersects the affine span of
  $S$ orthogonally. To compute the isotropy subgroups in general, we
  then only have to intersect the isotropy subgroups of the components
  of $x$ corresponding to a splitting of $V_0$ into $K$-irreducible
  subspaces.
\end{remark}

To investigate the isotropy groups of points in $V_0$, we only have to
consider the unit sphere of $V_0$. Recall that we choose the
$G$-invariant inner product such that for
$x \in V_0 = \bigoplus_{\alpha \in S} \C_\alpha$ we obtain the norm
$\norm{x}^2 = \sum_{\alpha \in S}\abs{x_{\alpha}}^2$. Thus the momenta
of the elements of the unit sphere of $V_0$ are contained in the convex hull of
$S$.

Now we want to construct examples of representations $V^c$ and
subspaces $V_0\simeq \bigoplus_{\alpha \in S} \C_\alpha$ with $S$ as above with
the property that $(V_0)_\tau$ contains points of different isotropy
types. Then there has to be a point $x \in V_0$ with $G_x \ne T_x$ and
thus $G_\mu \ne T$ for $\mu = \J(x)$. Thus $\mu$ has to be contained
in a Weyl wall.

Thus we search weights $\alpha_1, \dots, \alpha_l$ such that there is
an isotropy subspace $(\LT^*)^K$ such that the orthogonal projections
of $\alpha_1, \dots, \alpha_l$ to $(\LT^*)^K$ coincide and the
intersection of $(\LT^*)^K$ with the interior of the convex hull of
$\alpha_1, \dots, \alpha_l$ is non-empty (and hence consists of one
point). W.~l.~o.~g we assume that $K$ is maximal with this
property. Let $\mu$ be the intersection point. Then $G_\mu = K$. By
Remark~\ref{rem:irreducible-components}, we assume that $V_0$ is
contained in the $G_\mu$-irreducible space $\tilde V_0$.
\cite[Chapter V, (8.1)]{btd} implies that $G_\mu$ has a connected
covering group of the form $C \times Z$, where $Z$ is the identity
component of the center of $G_\mu$ and $C$ is semi-simple. The
Lie algebra of the group $Z$ coincides with $\LT^{G_\mu}$ and
$\LG_\mu = \mathfrak c \oplus \LT^{G_\mu}$, where $\mathfrak c$
denotes the Lie algebra of $C$.  We can consider $\tilde V_0$ as a
$C$-representation. Since
$\tilde \LT:=(\LT^{G_\mu})^\perp \subset \LT$ is the Lie algebra of a
maximal torus of $C$, the weights of the $C$-representation
$\tilde V_0$ are given by their projections to the dual Lie algebra
$\tilde \LT^*$. The projection of $\mu$ to $\tilde \LT^*$ is
$0$. Hence the projections $\bar \alpha_1, \dots, \bar \alpha_l$ of
$\alpha_1, \dots, \alpha_l$ form a linearly dependent set $\bar S$,
but $\bar S$ is still independent in the affine sense: No point is
contained in the affine span of the others. Moreover, $\bar S$ is
maximal within its affine span. Consider the two following examples
of $C$-irreducible representations with subsets of this type:

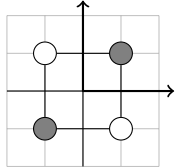
\begin{figure}[htbp]
  \centering
  \begin{tikzpicture}[every node/.style={circle,draw,fill =
        gray, inner sep = 0pt,
        minimum size =3mm, text = black}]
      \draw[very thin, lightgray, step =.5cm] (-1,-1) grid (1,1);
      \draw (-1,0) -- (0,0) -- (0,-1);
      \draw[thick,->] (0,0) -- (1.2,0);
      \draw[thick,->] (0,0) -- (0,1.2);
     \node[fill = white] (a) at (-.5,.5) {};
    \node (b) at (.5,.5) {};
    \node (c) at (-.5,-.5) {};
    \node[fill = white] (d) at (.5,-.5) {};
    \draw (a) edge[-] (b)
    edge[-] (c);
    \draw (d) edge[-] (b)
    edge[-] (c);
  \end{tikzpicture}
  \caption{$C= \SU(2) \times \SU(2)$, maximal weight $(1,1)$}
  \label{fig:c4}
\end{figure}

\begin{example}\label{ex:square}
  Suppose that $C=\SU(2)\times\SU(2)$ and consider the $C$-irreducible
  representation with maximal weight $(1,1)$ (with respect to the
  coordinate system corresponding to the basis given in
  Section~\ref{sec:su3} on each factor on the Lie algebra
  $\R \times \R$), see the weight diagram in Figure~\ref{fig:c4}. Then
  the weights $(-1,1)$ and $(1,-1)$ (depicted in white in
  Figure~\ref{fig:c4}) contain $0$ in their affine span.
\end{example}

\newlength{\uu}%
\newcommand{\mytzgrid}{%
  \draw[thin] (-3,0,0)--(3,0,0);
  \draw[thin] (0,-3,0)--(0,3,0);
  \draw[thin] (0,0,-3)--(0,0,3);
  \draw[thick,->] (0,0,0)--(3.2,0,0);
  \draw[thick,->] (0,0,0)--(0,3.2,0);
  \foreach \i in {1,2,3}{ \draw[very
    thin, lightgray] (-3,\i,0)--(3-\i,\i,0); \draw[very thin,
    lightgray] (-3+\i,-\i,0)--(3,-\i,0); \draw[very thin, lightgray]
    (0,-3,\i)--(0,3-\i,\i); \draw[very thin, lightgray]
    (0,-3+\i,-\i)--(0,3,-\i); \draw[very thin, lightgray] (\i,
    0,-3+\i)--(\i,0,3); \draw[very thin, lightgray] (-\i,
    0,-3)--(-\i,0,3-\i); } }%
\newcommand{\mytzcVI}{%
  \mytzgrid
  \draw[thin] (2,0,0)--(0,-2,0)--(0,0,2)--cycle;
  \draw[thin] (-1,0,0)--(0,1,0)--(0,0,-1)--cycle;
  \node[fill=white] (a) at (-1,0,0) {};
  \node (c) at (0,1,0) {} edge [-] (a);
  \node  (d) at (0,0,-1) {} edge [-] (a);
  \node[fill=white] (b) at (2,0,0) {} edge [-] (c)
  edge [-] (d);
  \node at (0,-2,0) {};
  \node at (0,0,2) {};
}

  \setlength{\uu}{.8cm}%
  \begin{figure}
    \centering
    \begin{tikzpicture}[x=\uu,y={(.5\uu,.866025\uu)},%
    z={(-.5\uu,.866025\uu)}, every node/.style={circle,draw,fill =
        gray, inner sep = 0pt,
        minimum size =4mm, text = black}]
      \mytzcVI
  \end{tikzpicture}%
    \caption{$C= \SU(3)$, maximal weight $(2,0)$}
    \label{fig:c6}
    \end{figure}
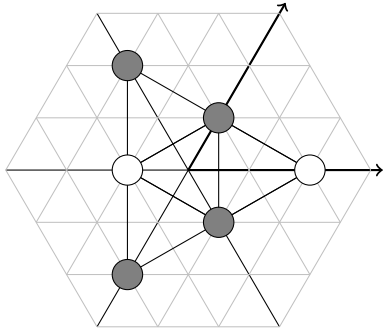
    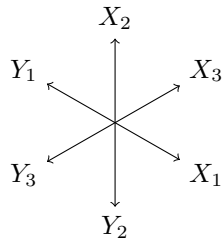
\begin{figure}
      \centering
      \begin{tikzpicture}[x=\uu,y={(.5\uu,.866025\uu)},%
        z={(-.5\uu,.866025\uu)}]
          \node (X1) at (1,0,-1) {$X_1$};
          \node (X2) at (0,1,1) {$X_2$};
          \node (X3) at (1,1,0) {$X_3$};
          \node (Y1) at (-1,0,1) {$Y_1$};
          \node (Y2) at (0,-1,-1) {$Y_2$};
          \node (Y3) at (-1,-1,0) {$Y_3$};
          \draw[<->] (X1) edge (Y1);
          \draw[<->] (X2) edge (Y2);
          \draw[<->] (X3) edge (Y3);
        \end{tikzpicture}
        \caption{Action of the root vectors of $\SU(3)$}
        \label{fig:A3}
  \end{figure}

\begin{example}\label{ex:triangle}
  Let $C$ be the group $\SU(3)$ and consider the $C$-irreducible
  representation with maximal weight $(2,0)$ (again, with respect to
  the coordinate system corresponding to the basis given in
  Section~\ref{sec:su3}), see Figure~\ref{fig:c6}. Then $0$ is
  contained in the affine span of the two weights $(2,0)$ and
  $(-1,0)$.
\end{example}

In both examples, the sets $\bar S$ have an important additional
property: We call a weight a \emph{neighbor} of another weight of an
irreducible representation of $K$ iff their difference coincides with
a root of $K$. In our examples, $\bar S$ consists of two weights which
have common neighbors. This makes possible that the isotropy type of
points of the same stratum of $T$-relative equilibria differs on the
Lie algebra level: If $x_\alpha \in \CC_\alpha$, then
$\mathfrak k x_\alpha$ is contained in the span of $\CC_\alpha$ and
the weight spaces corresponding to neighbors of $\alpha$. To obtain
$x= (x_\alpha, x_\beta)\in \CC_\alpha \oplus \CC_\beta$ and
$\xi \in \mathfrak k$ with $\xi x = 0$, but $\xi x_\alpha \ne 0$ and
$\xi x_\beta \ne 0$, $\alpha$ and $\beta$ hence need to have at least
one common neighbor.

In the figures with weight diagrams, there is an edge between weights
iff they are neighbors. The common neighbors of the weights of $\bar
S$ are given by $(1,1)$ and $(-1,-1)$ in Example~\ref{ex:square} and
by  $(0,1)$ and $(1,-1)$ in Example~\ref{ex:triangle}.

In the following, we will discuss some examples of representations for
which we obtain strata of $T$-relative equilibria with different
isotropy Lie algebras. Examples~\ref{ex:square} and \ref{ex:triangle}
occur as subspaces of these spaces.

We will restrict our
investigation to the computation of isotropy Lie algebras. As pointed
out in remark~\ref{rem:lie-algebra-strata}, this often suffices to
determine the dimensions of the strata and indeed this will be the
case for our example strata.

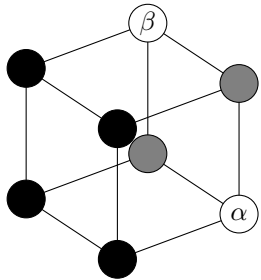
\begin{figure}[htbp]
  \centering
  \begin{tikzpicture}[x={(1.6cm,.6cm)},y={(1.2cm,-.8cm)},
    z={(0cm,1.732cm)}, every
    node/.style={circle,draw, fill = black, inner sep = 0pt,
      minimum size =5mm, text = black}]
      \draw (0,0,0) -- (1,0,0) -- (1,1,0) -- (0,1,0)-- cycle;
      \draw (0,0,1) -- (1,0,1) -- (1,1,1) -- (0,1,1)--cycle;
    \draw (0,0,0) -- (0,0,1);
    \draw (1,0,0) -- (1,0,1);
    \draw (0,1,0) -- (0,1,1);
    \draw (1,1,0) -- (1,1,1);
    \node [fill=white] at (1,1,0) {$\alpha$};
    \node [fill=white] at (1,0,1) {$\beta$};
    \node [fill=gray] at (1,0,0) {};
    \node [fill=gray] at (1,1,1) {};
    \node at (0,0,0)  {};
    \node at (0,0,1)  {};
    \node at (0,1,0)  {};
    \node at (0,1,1)  {};
  \end{tikzpicture}%
  \caption{$G= \SU(2)\times \SU(2) \times \SU(2)$, maximal weight $(1,1,1)$}
  \label{fig:c8}
\end{figure}

\begin{example}\label{ex:cube}
  Suppose that $G= \SU(2) \times \SU(2) \times \SU(2)$ and hence
  \begin{equation*}
    \LG= \su(2) \times \su(2) \times \su(2).
  \end{equation*}
  Let $V_1\simeq \CC^2$ denote the standard representation of $\SU(2)$. 
  Suppose that
  \begin{equation*}
    V^c \simeq V_1 \otimes V_1 \otimes V_1.
  \end{equation*}
  such that the $i$-th factor of $\SU(2)
  \times \SU(2) \times \SU(2)$ acts on the $i$-th factor of the tensor
  product, i. e.
  \begin{equation*}
    (g_1,g_2,g_3)( v_1\otimes v_2 \otimes v_3) = (g_1 v_1) \otimes (g_2
    v_2) \otimes (g_3 v_3).
  \end{equation*}
  The $\SU(2)$-representation $V_1$ has the weights $1$ and $-1$ with
  weight spaces spanned by $e_1$ and $e_2$ respectively (when we
  identify the maximal torus of $\SU(2)$ with $\quot \R \Z$ and its
  Lie algebra with $\R$, see Section~\ref{sec:su3}). Hence the weights of
  the representation $V^c$ are of the form $((-1)^i, (-1)^j, (-1)^k)$
  with $i, j, k \in \{0,1\}$. The Weyl walls coincide with the planes
  spanned by two coordinate axes.

  Now we consider the space
  \begin{equation*}
    V_0
    = \CC_{(1,1,-1)} \oplus \CC_{(1,-1,1)}=\langle e_1\otimes e_1
    \otimes e_2, e_1 \otimes e_2 \otimes e_1 \rangle . 
  \end{equation*}
  Since $\alpha:=(1,1,-1)$ and
  $\beta:=(1,-1,1)$ are linearly independent and are the maximal subset of
  weights within their affine span, there is a branch of relative
  equilibria tangent to $V_0$. The points $(x_\alpha, x_\beta)$ with
  $|x_\alpha| \ne 0 \ne |x_\beta|$ form the stratum $(V_0)_\tau$ of $V_0$ of
  minimal isotropy with respect to the $T$-action. Since
  \begin{equation*}
    \J(x_\alpha,x_\beta) = \pi |x_\alpha|^2\alpha + \pi
    |x_\beta|^2\beta = \pi((|x_\alpha|^2+  |x_\alpha|^2,
    |x_\alpha|^2- |x_\alpha|^2, |x_\beta|^2 -|x_\alpha|^2),
  \end{equation*}
  $\J(x_\alpha,x_\beta)$ is not contained in any Weyl wall if
  $|x_\alpha| \ne |x_\beta|$. In this case,
  $\LG_{\J(x_\alpha,x_\beta)}= \LT$ and hence $\LG_x = \LT_x$. If
  $|x_\alpha| = |x_\beta|$ in contrast, $\J(x_\alpha,x_\beta)=:\mu$ is
  contained in the intersection of two Weyl walls: the ones
  corresponding to the pair of roots $(0,\pm 2, 0)$ and to
  $(0,0, \pm 2)$ respectively.  Due to Lemma~\ref{lem:BTD},
  $\LG_\mu = \LT\oplus M_{(0,2,0)} \oplus M_{(0,0, 2)}$, where
  $M_{\alpha}$ denotes the real part of the sum of the weight spaces
  for $\pm \alpha$ in $\LG \otimes \C$ respectively. Thus we then
  obtain
   \begin{equation*}
     \LG_\mu = \setthat{
       \begin{pmatrix}
         A & &\\
         & B &\\
         & & C
       \end{pmatrix}}
    {A =
       \begin{pmatrix}
         a\I & 0\\
         0 & -a\I
       \end{pmatrix},
       a \in \RR,   {B, C \in \su(2)}}.
   \end{equation*}
   We know that $\LG_x \subseteq \LG_\mu$. The intersection
   $\LT \cap \LG_x = \LT_x$ is constant on the stratum
   $(V_0)_\tau = \setthat {(x_\alpha, x_\beta) \in V_0}{x_\alpha\ne 0,
     x_\beta \ne 0}$: For $x \in (V_0)_\tau$, we have
   $\LT_x = \langle (0, 1,1) \rangle$.

   To compute $\LG_x$ in the case $|x_\alpha| = |x_\beta|$, we
   investigate the action of the complexification
   $\LG_\mu\otimes \C \simeq \C\times \LSL(2,\CC) \times
   \LSL(2,\CC)$. We use the basis of $\LSL(2,\CC)$ given in
   \cite{hall2003lie}:
   \begin{equation*}
     X =
     \begin{pmatrix}
       0&1 \\
       0&0
     \end{pmatrix},
     Y =
     \begin{pmatrix}
       0&0\\
       1&0
     \end{pmatrix},
     H=
     \begin{pmatrix}
       1&0\\
       0&-1
     \end{pmatrix}.
   \end{equation*}
   Then a basis of
   $\LG_\mu\otimes \C \subset \LSL(2,\C) \times \LSL(2,\C)\times
   \LSL(2, \C)$ is given by $H_1:= (H, 0,0)$, $H_2:= (0,H,0)$,
   $H_3:=(0,0,H)$, $X_2:= (0,X,0)$, $X_3:= (0,0,X)$, $Y_2:= (0,Y,0)$,
   and $Y_3:= (0,0,Y)$. We search for elements
   $Z \notin \LT\otimes \C = \langle H_1, H_2, H_3 \rangle$ with
   $Zx=0$ for $x \in V_0$ with $|x_\alpha| = |x_\beta| \ne 0$. While
   the elements of $\LT \otimes \C$ preserve the weight spaces, the
   elements $X_2$, $X_3$, $Y_2$, and $Y_3$ are root vectors and hence
   cause shifts between the weight spaces. In more detail,
   \begin{align*}
     X_2(e_1\otimes e_1 \otimes e_2) = 0, \quad & X_2(e_1\otimes e_2
     \otimes e_1) = e_1 \otimes e_1 \otimes e_1\\
     X_3(e_1\otimes e_1 \otimes e_2) = e_1 \otimes e_1 \otimes e_1,
     \quad & X_3(e_1\otimes e_2
     \otimes e_1) = 0\\
     Y_2(e_1\otimes e_1 \otimes e_2) = e_1 \otimes e_2 \otimes e_2,
     \quad &
     Y_2(e_1\otimes e_2 \otimes e_1) = 0\\
     Y_3(e_1\otimes e_1 \otimes e_2) = 0,
     \quad & Y_3(e_1\otimes e_2 \otimes
     e_1) = e_1 \otimes e_2 \otimes e_2.
   \end{align*}
   Thus the space $V_0$ is contained in the $\LG_\mu$-invariant subspace
   \begin{equation*}
     \tilde V_0 = \langle <e_1\otimes e_1 \otimes e_1,\, e_1 \otimes e_1
     \otimes e_2\,, \e_1 \otimes e_2 \otimes e_1,\, e_1 \otimes e_2
     \otimes e_2>.
   \end{equation*}
   See Figure~\ref{fig:c8} for a weight diagram of $V^c$. The weights
   of $V_0$ are depicted in white, the other ones of $\tilde V_0$ in gray.
   
   We recognize Example~\ref{ex:square}: We can regard $\tilde V_0$ as
   a representation of the Lie algebra
   $ \LSL(2, \C) \times \LSL(2, \C) $, the complexified Lie algebra of
   $\SU(2) \times \SU(2)$ identified with the second and third
   component of $G$. Then the weights of $\tilde V_0$ are simply
   $(1,1)$, $(1,-1)$, $(-1,1)$ and $(-1,-1)$ and $V_0$ is spanned by
   the weight spaces of $(1,-1)$ and $(-1,1)$.
 
   Suppose that $Z x = 0$ for some $Z \in \LG_\mu\otimes \C$ and a
   given $x= (x_\alpha, x_\beta) \in V_0$ with $x_\alpha \ne 0$. Split
   $Z$ into $Z = Z_T + Z'$ with $Z_T \in \LT$ and
   $Z' \in \langle X_2, X_3, Y_2, Y_3 \rangle$. Then
   $Z_TV_0 \subset V_0$, while $Z'V_0$ is contained in
   $\langle e_1 \otimes e_1 \otimes e_1, e_1 \otimes e_2 \otimes e_2
   \rangle$, which is a complement of $V_0$ within $\tilde V_0$. Thus
   we have $Z_T x = 0$ and $Z' x =0$. Moreover, if we split $Z'$
   further into $Z' = Z_X + Z_Y$ with
   $Z_X \in \langle X_2, X_3 \rangle$ and
   $Z_Y \in \langle Y_2, Y_3 \rangle$, then
   $Z_Xx \in \langle e_1 \otimes e_1 \otimes e_1 \rangle$ and
   $Z_Y \in \langle e_1 \otimes e_2 \otimes e_2 \rangle$. Thus
   $Z_X x = 0$ and $Z_Y x =0$. Suppose that $Z_X = a X_2 +b X_3$ and
   $Z_Y = c Y_2 + d Y_3$, $a, b, c, d \in \C$. Then
   \begin{align*}
     bx_\alpha + ax_\beta = 0\\
     cx_\alpha + dx_\beta = 0.
   \end{align*}
   The solutions $a, b, c, d$ correspond to elements of the isotropy Lie
   subalgebra of the complexified Lie algebra $\LG_\mu\otimes\C$. To
   determine $\LG_x$, we have to identify the real solutions, i. e. the
   solutions in $\LG_\mu$. $Z$ is contained in $\LG_\mu$ iff this holds
   for $Z_T$ and $Z'$. Furthermore, $Z' \in \LG_\mu$ is equivalent to
   $c = -\bar a$ and $d = -\bar b$. Thus for a fixed
   $Z' \in \LT^\perp\subset \LG_\mu$ there is a non-trivial solution
   $x=(x_\alpha, x_\beta) \in V_0$ of $Z'x = 0$ iff
   \begin{equation*}
     \det
     \begin{pmatrix}
       b & a\\
       -\bar a & -\bar b
     \end{pmatrix}
     = |a|^2 - |b|^2 = 0.
   \end{equation*}
   If this condition is satisfied, the kernel is spanned by
   $(a, -b)^T$.  Hence we obtain again the necessary condition
   $|x_\alpha| = |x_\beta|$ from above.  Conversely, if
   $|x_\alpha| = |x_\beta|$, there is a non-zero $Z'$ in the
   complement of $\LT$ within $\LG_\mu$ that satisfies $Z'x =
   0$. Explicitly, the solutions of $Z'$ are given by the vectors $(a, b)$
   in the complex span of $\langle x_\alpha, - x_\beta \rangle$ and
   the requirement $a = -\bar c $ and $b = -\bar d$. Thus, the points
   $x$ with $|x_\alpha| =|x_\beta|$ have a 3-dimensional isotropy Lie
   algebra $\LG_x$ of rank 1, which is hence isomorphic to
   $\su(2)$. In contrast, the isotropy Lie algebra of the other points
   of $(V_0)_\tau$ is just $\LG_x = \LT_x = \langle(0,1,1)\rangle$.

   Thus the $G$-orbit of the stratum $m_{V_0}((V_0)_\tau)$ of
   $T$-relative equilibria consists of two strata: Since
   $G_x \subset T$ for every point $x$ of $(V_0)_\tau$ with
   $\abs{x_\alpha} \ne \abs{x_\beta}$ and
   $(\LT') ^T=\LT'= (\LT_x)^\perp \subset \LT$ is 2-dimensional, the
   points whose isotropy Lie algebra is conjugated to
   $\langle(0,1,1)\rangle$ form a manifold of dimension
   \begin{equation*}
     \dim \LG - \dim \langle(0,1,1)\rangle + \dim \LT' = 9 -1 + 2 = 10.
   \end{equation*}
   For $x \in (V_0)_\tau $ with $\abs{x_\alpha} = \abs{x_\beta}$, the
   momentum $\mu$ is contained in
   $\langle (1,0,0)\rangle = (\LT^*)^{G_\mu}\cong (\LT')^{G_\mu}$ and
   $G_\mu$ is the minimal coadjoint isotropy group which contains
   $G_x$. Thus the stratum of points with isotropy Lie algebras
   isomorphic to $\su(2)$ has the dimension
   \begin{equation*}
     \dim \LG - \dim \LG_x + \dim (\LT')^{G_\mu} = 9 -3 + 1 = 7.
   \end{equation*}
 \end{example}
 \begin{figure}[htbp]
   \centering
   \begin{tikzpicture}[x={(.8cm,.3cm)},y={(.6cm,-.4cm)}, z={(0cm,.866cm)},
     every node/.style={circle,draw,fill = black, inner sep = 0pt,
       minimum size =5mm, text = white}]
     \coordinate (a) at (0,0,4.2426);
     \coordinate (b) at (xyz polar cs:angle=30,radius=3);
     \coordinate (c) at (xyz polar cs:angle= -90,radius=3);
     \coordinate (d) at (xyz polar cs:angle=150,radius=3);
     \draw (a) -- (b) -- (c) -- (a);
     \draw (a) -- (d);
     \draw (b) -- (d);
     \draw (c) -- (d);
     \node[fill=white] at (a) {};
     \node[fill=gray] at (b) {};
     \node at (c) {};
     \node[fill=gray] at (d) {};
     \node[fill=gray] (e) at  (barycentric cs:a=1,b=1) {};
     \node (f)  at (barycentric cs:a=1,c=1) {};
     \node[fill=gray] (g)  at (barycentric cs:a=1,d=1) {};
     \node (h)  at (barycentric cs:b=1,c=1) {};
     \node[fill=white] (i)  at (barycentric cs:b=1,d=1) {};
     \node (j)  at (barycentric cs:c=1,d=1) {};
     \draw (e) edge (f) edge (g) edge (h) edge (i);
     \draw (f) edge (i) edge (h) edge (g) edge (j);
     \draw (j) edge (g) edge (h);
     \draw (i) edge (j) edge (h) edge (g);
   \end{tikzpicture}
   \caption{$G=\SU(4)$, maximal weight $(2,0,0)$}
   \label{fig:c10}
 \end{figure}
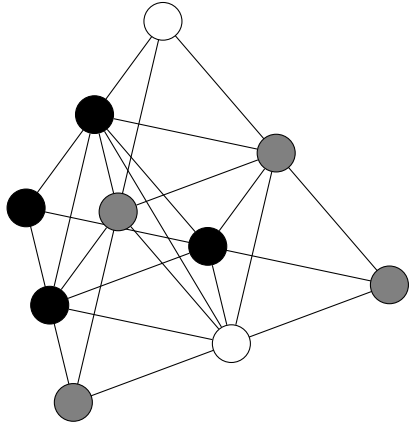
 \begin{example}\label{ex:pyramid}
   Suppose that there is a point $x_0 \in V_0$ with momentum $\mu$
   such that $G_\mu$ is covered by a group $\SU(3)\times T'$, where
   $T'$ is a torus, and that $V_0$ is contained in an irreducible
   $G_\mu$-subrepresentation $\tilde V_0$ of $V^c$ such that the
   situation is as in Example~\ref{ex:triangle}. More precisely, we
   assume that $\tilde V_0$ regarded as an $\SU(3)$-representation is
   irreducible with maximal weight $(2, 0)$ and $V_0$ is given by
   the sum of the two weight spaces corresponding to the weights
   $(2,0)$ and $(-1,0)$. We will see that, as in
   Example~\ref{ex:cube}, the main part of the calculation is
   independent of the actual embedding of $\tilde V_0$ into
   $V^c$. Thus we will start with the calculation of the $\su(3)$-part
   of the isotropy Lie algebras. Afterwards, we will discuss a
   concrete example of a $G$-representation $V^c$ with the
   $\SU(3)$-subrepresentation $\tilde V_0$.

   Again, we study the action of the complexified Lie algebra
   $\SU(3) \otimes \C \simeq \LSL(3,\C)$. We adopt the notation of
   \cite{hall2003lie}: We consider the following basis of
   $\LSL(3,\C)$:
   \begin{align*}
     H_1 =
     \begin{pmatrix}
       1&0&0\\
       0&-1&0\\
       0&0&0
     \end{pmatrix}, \quad
        &H_2 =
          \begin{pmatrix}
            0&0&0\\
            0&1&0\\
            0&0&-1
          \end{pmatrix},\\
     X_1 =
     \begin{pmatrix}
       0&1&0\\
       0&0&0\\
       0&0&0
     \end{pmatrix}, \quad
        &X_2 =
          \begin{pmatrix}
            0&0&0\\
            0&0&1\\
            0&0&0
          \end{pmatrix}, \quad
                 X_3 =
                 \begin{pmatrix}
                   0&0&1\\
                   0&0&0\\
                   0&0&0
                 \end{pmatrix},\\
     Y_1 =
     \begin{pmatrix}
       0&0&0\\
       1&0&0\\
       0&0&0
     \end{pmatrix}, \quad
        &Y_2 =
          \begin{pmatrix}
            0&0&0\\
            0&0&0\\
            0&1&0
          \end{pmatrix}, \quad
                 Y_3 =
                 \begin{pmatrix}
                   0&0&0\\
                   0&0&0\\
                   1&0&0
                 \end{pmatrix}.
   \end{align*}
   $H_1$ and $H_2$ form a basis of $\LT$ and the vectors $X_i$, $Y_i$
   are root vectors. If $X_i$ corresponds to the root $\rho_i$, then
   $Y_i$ corresponds to $-\rho_i$. With respect to the dual basis of
   $H_1$ and $H_2$, we have
   \begin{equation*}
     \rho_1 = (2, -1), \quad \rho_2 = (-1,2), \quad 
     \rho_3 = \rho_1 + \rho_2 = (1,1).
   \end{equation*}
   If $x \in \tilde V_0$ is a weight vector with the weight $\alpha$,
   then $X_ix$ either is a weight vector corresponding to
   $\alpha+\rho_i$ or $\alpha+\rho_i$ is not a weight of $\tilde V_0$
   and $X_ix=0$. Similarly, if $\alpha-\rho_i$ is a weight of
   $\tilde V_0$, then $Y_ix$ is a weight vector corresponding to
   $\alpha-\rho_i$. In particular, if we consider the weight spaces of
   $V_0 \simeq \C_{(2,0)}\oplus\C_{(-1,0)}$, we obtain
   \begin{align*}
     &\begin{matrix}
       X_1\left(\C_{(2,0)}\right)=\{0\},
       &X_2\left(\C_{(2,0)}\right)=\{0\},
       & X_3\left(\C_{(2,0)}\right)=\{0\},\\
       Y_1\left(\C_{(2,0)}\right)= \C_{(0,1)},
       & Y_2\left(\C_{(2,0)}\right)= \{0\},
       & Y_3\left(\C_{(2,0)}\right)= \C_{(1,-1)},\\                              
     \end{matrix}
     \\
     &\begin{matrix}
       X_1\left(\C_{(-1,0)}\right)=\C_{(1,-1)},
       &X_2\left(\C_{(-1,0)}\right)=\C_{(-2,2)},
       & X_3\left(\C_{(-1,0)}\right)=\C_{(0,1)},\\
       Y_1\left(\C_{(-1,0)}\right)= \{0\},
       & Y_2\left(\C_{(-1,0)}\right)= \C_{(0,-2)},
       & Y_3\left(\C_{(-1,0)}\right)= \{0\},\\                                
     \end{matrix}
   \end{align*}
   see Figures~\ref{fig:c6} and \ref{fig:A3}.

   The elements $X_2, Y_2$ act trivially on $\C_{(2,0)}$, but not on
   $\C_{(-1,0)}$ and the span of $X_2(\C_{(-1,0)})$ and
   $Y_2(\C_{(-1,0)})$ intersects the space $\LSL(3,\C)\left( \C_{(2,0)}\right)$
   only in $0$. Hence, if $Zx= 0$ for some $x \in (V_0)_\tau$ and
   $Z = Z_T + Z'$ with $Z_T \in \LT = \langle H_1, H_2\rangle$ and
   $Z' \in \langle X_1,X_2,X_3, Y_1, Y_2, Y_2 \rangle$, then $Z'$ has
   to be a linear combination of elements $X_1, X_3, Y_1, Y_3$.

   To determine solutions $Z',x$ of $Z'x=0$, we represent the linear
   maps $X_1, X_3, Y_1, Y_3$ with respect to an orthonormal basis of
   $\tilde V_0$.

   The $\SU(3)$-representation $\tilde V_0$ has maximal weight $(2,0)$
   and is hence isomorphic to the $\SU(3)$-representation on the space
   of homogeneous polynomials of degree 2 in 3 complex variables, see
   \cite{fulton_1996}. Thus an explicit description of the action is
   well-known and clear. However, we will follow another approach,
   which is also suitable for more complicated representations: We fix
   an $\SU(3)$-invariant Hermitian inner product. Since it is in
   particular $T$-invariant, the weight spaces are orthogonal to each
   other.  Thus we obtain an orthogonal basis of the $\C$-vector space
   $\tilde V_0$ if we choose one nontrivial subvector of every weight
   space. To obtain an orthonormal basis, we only have to determine
   the normalization factors.

   Now we consider the $3$-dimensional subspace
   $\C_{(2,0)}\oplus \C_{(0,1)} \oplus \C_{(-2,2)}$. This space is an
   irreducible representation of the Lie algebra generated by $H_1$,
   $X_1$, and $Y_1$, which is isomorphic to
   $\LSL(2,\C) \simeq \su(2)\otimes \C$. Hence it can be regarded as a
   3-dimensional complex representation of $\SU(2)$. Thus it is
   isomorphic to the one described in \cite[Section 4.2, Example
   4.10]{hall2003lie}, which consists of the homogeneous polynomials
   of degree 2 on $\C^2$. The weight spaces are spanned by the
   monomials $z_1^2$, $z_1z_2$, and $z_2^2$ with weights $-2$, $0$,
   and $2$ respectively. We have
   \begin{equation*}
     \begin{matrix}
       X_1(z_1^2) = -2z_1z_2,&
       X_1(z_1z_2) = - z_2^2,& X_1(z_2^2) = 0,\\
       Y_1(z_1^2) =0, &
       Y_1(z_1z_2) = - z_1^2,& Y_1(z_2^2) = -2 (z_1z_2).
     \end{matrix}
   \end{equation*}
   With respect to any orthonormal basis for a
   $\SU(2)$-invariant inner product, the action of $X_1 -Y_1$ is
   represented by a skew-Hermitian matrix. For the basis
   $z_1^2$,$z_1z_2$, and $z_2^2$, we obtain
   \begin{equation*}
     X_1-Y_1 =
     \begin{pmatrix}
       0&1&0\\
       -2&0&2\\
       0&-1&0
     \end{pmatrix}
   \end{equation*}
   Thus for an orthonormal basis $\lambda_1z_1^2$, $\lambda_2z_1z_2$, and
   $\lambda_3z_2^2$ with $\lambda_1,\lambda_2, \lambda_3 >0$, the matrix is
   \begin{equation*}
     X_1-Y_1 =
     \begin{pmatrix}
       0&\frac{\lambda_2}{\lambda_1}&0\\
       -2\frac{\lambda_1}{\lambda_2}&0&2\frac{\lambda_3}{\lambda_2}\\
       0&-\frac{\lambda_2}{\lambda_3}&0
     \end{pmatrix}.
   \end{equation*}
   This is skew-Hermitian for $\frac{\lambda_2}{\lambda_1} =
   \frac{\lambda_2}{\lambda_3} = \sqrt 2$. In particular,
   $\norm{Y_1(z)} = \sqrt{2} \norm{z}$ for $z \in
   \C_{(2,0)}$.

   Analogously we deduce that $\norm{Y_3(z)} = \sqrt{2} \norm{z}$ for
   $z \in \C_{(2,0)}$, $\norm{X_1(z)} = \norm{z}$ for
   $z \in \C_{(-1,0)}$, and $\norm{X_3(z)} = \norm{z}$ for
   $z \in \C_{(-1,0)}$. Since $[Y_1,Y_3]=0$, we can choose orthonormal
   vectors of the form
   \begin{align*}
     e_{(2,0)} \in \C_{(2,0)}, \quad
     e_{(0,1)} = \frac 1 {\sqrt 2} Y_1 e_{(2,0)}, \quad
     e_{(1,-1)} = \frac 1 {\sqrt 2} Y_3 e_{(2,0)},\\
     e_{(-1,0)} = Y_1  e_{(1,-1)}  = Y_3 e_{(0,1)}.
   \end{align*}
   Then a linear map
   $aX_1-\bar a Y_1 + b X_3 - \bar b Y_3$, $a, b \in \C$ from
   $V_0 = \C_{(-1,0)} \oplus \C_{(2,0)}$ to
   $\C_{(0,1)} \oplus \C_{(1,-1)}$ is represented by the matrix
   \begin{equation*}
     \begin{pmatrix}
       b & -\sqrt{2}\bar a \\
       a & -\sqrt{2} \bar b
     \end{pmatrix}.
   \end{equation*}
   Since the determinant is $\sqrt 2 (\abs{a}^2 - \abs{b}^2)$, the
   kernel is non-trivial iff $\abs{a} = \abs{b}$. If
   $(x_{(-1,0)}, x_{(2,0)})$ is a kernel element, we thus have
   $\abs{x_{(-1,0)}} = \sqrt{2} \abs{x_{(2,0)}}$, which is satisfied
   iff the momentum of $x$ with respect to the $\SU(3)$-action
   vanishes. Vice versa, if $x = (x_{(-1,0)}, x_{(2,0)})$ satisfies
   this condition and $x_{(-1,0)} = \sqrt 2 \theta x_{(2,0)}$, then
   for any $a \in \C$, $x$ is in the kernel of the above matrix iff
   $b := \overline{\theta a}$. Therefore $\su(3)_x$ is a 3-dimensional
   Lie subalgebra of the real Lie algebra $\su(3)$. Again, since
   $\su(3)_x$ is the Lie algebra of the compact group $\SU(3)_x$ and
   since $\su(3)_x$ is simple, $\su(3)_x \simeq \su(2)$.
   
   We now give an example of a group $G$ and a $G$-representation
   $V^c$ such that generically, if the center space of $\der X_h(0)$
   is isomorphic to $V^c$, there is a $\xi$ in $\LT^K$ with
   $K \simeq \SU(3) \times T'$ for some torus $T'$ such that
   $\der^2(h-\J^\xi)(0) = \tilde V_0$ is of the above form: Set
   $G = \SU(4)$ and consider the basis of $\LT\otimes \C$ formed by
   the diagonal matrices $H_1$, $H_2$, and $H_3$ with diagonal entries
   $(1,-1,0,0)$, $(0,1,-1,0)$, and $(0,0,1,-1)$ respectively. Let
   $V^c$ be the representation with maximal weight $\lambda=(2,0,0)$
   with respect to the dual basis of $\{H_1, H_2, H_3\}$. Now we
   change the basis of $\LT$ and replace $H_3$ by the diagonal matrix
   with diagonal entries $(1,1,1,-3)$, which we denote by $H_3'$. With
   respect to the dual basis of $\{H_1, H_2, H_3'\}$, $\lambda$ is
   given by $(2,0,2)$.

   When we consider the complexified adjoint action, $H_3'$ is
   contained in the kernel of all elements of $\LSL(4,\C)$ of the form
   \begin{equation*}
     \begin{pmatrix}
       &&&\aug& 0\\
      &Z&&\aug& 0\\
      &&&\aug& 0\\
      \hline
      0&0&0&\aug& 0
    \end{pmatrix} =: \bar Z,
  \end{equation*}
  $Z \in \LSL(3,\C)$. The span of this subalgebra and $\LT\otimes \C$
  generates the complexified isotropy Lie algebra of $H_3'$ within
  $\LSL(4,\C)$, which is isomorphic to the Lie algebra
  $\LSL(3,\C) \oplus \C$: The isomorphism is given by
  \begin{equation*}
    (Z,z) \in \LSL(3,\C) \oplus \C \mapsto
    \bar Z + z H_3' \in \LSL(4,\C)_{H_3'}.
  \end{equation*}
  The intersection of the complex Lie algebra $\LSL(4,\C)_{H_3'}$ with
  $\LG$ is isomorphic to $\su(3) \oplus \R$ and hence the group
  $G_{H_3'}$ is isomorphic to $\SU(3) \times S^1$. (It is given by the
  elements
  \begin{equation*}
    \begin{pmatrix}
      &&&\aug& 0\\
      &U&&\aug& 0\\
      &&&\aug& 0\\
      \hline
      0&0&0&\aug& (\det U)^{-1}
    \end{pmatrix},
  \end{equation*}
  $U \in \U (3)$, of $\SU(4)$.)
  
  The dual vector $\nu=(0,0,1)$ of $H_3'$ has the same isotropy
  subgroup $G_\nu$. Moreover, the roots of $G_\nu$ are the vectors
  given by $(r_1, r_2,0)$ (with respect to the dual basis of
  $\{H_1, H_2, H_3'\}$) such that $(r_1, r_2)$ is a root of $\SU(3)$
  (with respect to the dual basis of $\{H_1, H_2\}$). Thus
  $(\alpha_1,\alpha_2,2)$ is a weight of $V^c$ iff
  $(\alpha_1, \alpha_2)$ is a weight of the $\SU(2)$-representation
  with maximal weight $(2,0)$. Moreover, all weights of $V^c$ occur
  with multiplicity $1$. (This follows from the representation theory
  of $\SU(n)$ given by Young tableaux, see \cite{fulton_1996}. $V^c$
  coincides with the representation of $\SU(4)$ on the space of
  homogeneous polynomials of degree 2 in four complex variables.)
  
  Thus the weight spaces corresponding to weights of the form
  $(\alpha_1,\alpha_2,2)$ span a subspace $\tilde V_0$ as described
  above.

  See Figure~\ref{fig:c10} for a weight diagram of $V^c$. The weights
  of $V_0=\C_{(2,0,2)}\oplus \C_{(-1,0,2)}$ are depicted in white, the
  other weights of $\tilde V_0$ in gray.

  For $x \in (V_0)_\tau$ with
  $\abs{x_{(-1,0,2)}} = \sqrt{2} \abs{x_{(2,0,2)}}$, the isotropy Lie
  algebra $\LG_x$ is isomorphic to $\su(2)$: The affine span of the
  weights of $\tilde V_0$ is orthogonal to $\langle \nu \rangle$ and
  intersects $\langle \nu \rangle$ in the single point $2\nu$. Since
  $\langle \nu \rangle =((\LT)^*)^{G_\mu}$, all isotropy subgroups of
  elements of $(V_0)_\tau$ are contained in $G_\nu$, see
  Lemma~\ref{lem:tildeV01}. We can identify the roots of
  $G_\nu \cong \SU(3)\times S^1$ with those of $\SU(3)$.

  Again, $(\LG_\nu)_x$ is given by the span of $\LT_x$ and the space
  of solutions $\xi$ of the equation $\xi x=0$ that are contained in
  the real part of the span of the root vectors. The solutions $\xi$
  are determined by the above calculation. Moreover, under the
  identification $(\LT)^* \cong \LT$ via a $T$-invariant inner
  product, the space $\LT_x$ is given by the orthogonal complement of
  the span of the weights of $V_0$ for any $x \in (V_0)_\tau$. Since
  the Lie algebra of the $S^1$-factor of $G_\nu$ is contained in the
  span of the weights of $V_0$, we also can identify $\LT_x$ with the
  isotropy Lie algebra of $x$ with respect to the Lie algebra of the
  corresponding maximal torus of $\SU(3)$. Hence $(\LG_\nu)_x$ is
  isomorphic to $\su(3)_x\cong \su(2)$.

  The stratum of points of this isotropy type has the dimension
  \begin{equation*}
    \dim G -\dim G_x +(\LT')^L = 15-3+1 = 13,
  \end{equation*}
  since $\J(x) \in \langle \nu\rangle$ and $L=L(G_x) = G_\nu$.
  
  For any $x \in (V_0)_\tau$ with
  $\abs{x_{(-1,0,1)}} \ne \sqrt{2} \abs{x_{(2,0,1)}}$, we obtain
  $\LG_x = \LT_x$, which is $1$-dimensional. Moreover, the isotropy
  group of
  \begin{equation*}
    \mu:=\J(x)
    = \sum \abs{x_{(-1,0,2)}}^2 (-1,0,2) +\abs{x_{(2,0,2)}}^2(-1,0,2)
  \end{equation*}
  is given by the connected subgroup of $G = \SU(4)$ whose Lie
  algebra is spanned by $\bar X_2$, $\bar Z_2$ and $\LT$. Hence
  $\C_{(-1,0,2)}$ and $\C_{(2,0,2)}$ are contained in different
  irreducible $G_\mu$-subrepresentations of $V^c$ and thus $G_x$
  coincides with $(G_\mu)_{(2,0,2)} \cap (G_\mu)_{(-1,0,2)}$. Hence all
  points of $(V_0)_\tau$ with $\abs{x_{(-1,0,1)}} \ne \sqrt{2}
  \abs{x_{(2,0,2)}}$ are of the same isotropy type and the isotropy
  subgroups are conjugated by elements of $T$. Thus $(\LT')^L = \LT'$
  and we obtain
  \begin{equation*}
    \dim G -\dim G_x +(\LT')^L = 15 -1+2 =16.
  \end{equation*}

  We also want to point out, that $G_x \subset G_\nu$ holds for all
  points $x \in (V_0)_\tau$ but not for all $x \in V_0$: If
  $x \in \C_{(2,0,2)}$ and $\mu := \J(x)$, then $G_\mu$ is a group
  isomorphic to $\U(3)\simeq \SU(3) \times S^1$ and not contained in
  $G_\nu$. Moreover, for every root $\rho$ of $G_\mu$ the form
  $(2,0,1) +\rho$ is not a weight of $V^c$ and hence $G_x = G_\mu$.
\end{example}
\begin{example}
  We now consider a configuration similar to that in
  Examples~\ref{ex:square} and \ref{ex:cube}: We consider the
  irreducible $\SU(4)$-representation $\tilde V_0$ with maximal weight
  $\alpha:=\frac 1 2(1,1,-1,-1)$, where we identify $\LT$ with the
  subspace of points $(t_1, t_2, t_3, t_4)$ of $\R^4$ with
  $t_1+t_2+t_3+t_4 = 0$. The weights of this representation are given
  by the Weyl group orbit of $\alpha$.

  We consider the two weights $\alpha$ and $\beta:=-\alpha$, which are
  linearly dependent but obviously independent in the affine sense.

  If $x=(x_\alpha, x_\beta) \in V_0:=\C_\alpha+\C_\beta$ and $\norm{x} =
  1$, then $\J(x) = 0$ iff $\abs{x_\alpha}= \abs{x_\beta}= \frac 1
  {\sqrt 2}$.

  The four other weights of the representation are all common
  neighbors of $\alpha$ and $\beta$. The roots of $\SU(4)$ are given
  by
  \begin{align*}
    \pm \rho_1:= \pm (1,-1,0,0), \quad &  \pm \rho_2:= \pm (0,0,1,-1),\\
    \pm \rho_3:= \pm (0,1,-1,0),\quad &  \pm \rho_4:= \pm (1,0,0,-1),\\
    \pm \rho_5:= \pm (1,0,-1,0),\quad &  \pm \rho_6:= \pm (0,1,0,-1).\\
  \end{align*}
  Let $X_i$ and $Y_i$ be the root vectors given by matrices with a single
  non-vanishing entry 1 corresponding to $\rho_i$ and $-\rho_i$
  respectively. Then the linear maps defined by $X_1$, $X_2$, $Y_1$,
  and $Y_2$ vanish on $\C_\alpha$ and $\C_\beta$. If
  $\abs{x_\alpha}\ne \abs{x_\beta}$, then $\J(x)$ is contained exactly
  in the Weyl walls corresponding to $\pm \rho_1$ and $\pm
  \rho_2$. Thus in this case, $\LSL(4,\C)_x$ is
  given by the sum of the 2-dimensional complex space $\LT_x\otimes \C$
  and the span of $X_1, X_2, Y_1,Y_2$ which is isomorphic to $\LSL(2,\C)
  \times \LSL(2,\C)$. Hence $\su(4)_x$ is isomorphic to $\su(2) \times
  \su(2)$.

  If $\abs{x_\alpha} =\abs{x_\beta}$, then $\J(x)=0$ and hence we have
  to consider the action of every root of $\SU(4)$. Again,
  $\LSL(4,\C)_x$ contains $\LT_x\otimes \C$
  and the span of $X_1, X_2, Y_1,Y_2$, but in addition a subspace of
  the span of $X_i,Y_i$, $i \in \{3,4,5,6\}$:

  Since the differences of one of the weights
  $\pm\frac 1 2(1,-1,1,-1)$ with one of the weights $\alpha, \beta$
  coincide with one of the roots $\pm \rho_3, \pm \rho_4 $, while the
  differences of one of the weights $\pm\frac 1 2(1,-1,-1,1)$ with one
  of the weights $\alpha, \beta$ coincide with one of the roots
  $\pm \rho_5, \pm \rho_6 $, we obtain two independent linear systems
  of linear equations similar to the linear system that we know from
  Example~\ref{ex:cube}, one involves the coefficients of
  $X_3, X_4, Y_3, Y_4$ and the other the ones of $X_5, X_6, Y_5,
  Y_6$. Thus $\su(4)_x$ is a 10-dimensional group of rank 2. Indeed,
  $\su(4)_x$ is isomorphic to $\mathfrak {so}(5)$. (To see this,
  compute the roots of $\su(4)_x$. The root system is of the type
  $B_2$, which is the class of the root system of $\mathfrak {so}(5)$,
  see \cite[Chapter 5]{btd}).

  When for example $G= \SU(5)$ and $V^c$ is the representation with
  maximal weight $\frac 1 5 (3,3,-2,-2,-2)$, then the affine set
  spanned by the weights $\frac 1 5 (3,3,-2,-2,-2)$ and $\frac 1 5
  (-2,-2,3,3,-2)$ contains the point $\frac 1 5 (1,1,1,1,-4)$, which is
  fixed by all reflections corresponding to roots with the last entry
  $0$. Thus its coadjoint isotropy group $K$ is generated by the image
  of the embedding $\SU(4) \to \SU(5)$, given by
  \begin{equation*}
    A \mapsto  \begin{pmatrix}
      &&&\aug& 0\\
      &A&&\aug& 0\\
      &&&\aug& 0\\
      \hline
      0&0&0&\aug& 1
    \end{pmatrix},
  \end{equation*}
  and the group of matrices $\exp(t\cdot \diag(1,1,1,1,-4))$,
  $t\in \R$, which is isomorphic to $S^1$. Since $((\LT)^*)^K$ is
  orthogonal to the affine span of
  $\bar \alpha=\frac 1 5 (3,3,-2,-2,-2)$ and
  $\bar \beta=\frac 1 5 (-2,-2,3,3,-2)$ and intersects it in a single
  point, the isotropy subgroups of points in $(V_0)_\tau$ are
  contained in $K$.  $V_0$ is contained in the space $\tilde V_0$
  spanned by the weight spaces corresponding to weights with last
  entry $-2$. Since the weights of $V^c$ are given by the elements of
  the Weyl group orbit of $\frac 1 5 (3,3,-2,-2,-2)$, $\tilde V_0$
  regarded as an $\SU(4)$-representation has maximal weight
  $\frac 1 2 (1,1,-1,-1)$.  Thus the above computation shows that for
  any point $x\in (V_0)_\tau$ with $\abs{x_\alpha} \ne \abs{x_\beta}$,
  the isotropy Lie algebra is isomorphic to $\su(2)\times \su(2)$ and
  generically the corresponding stratum of relative equilibria of the
  same isotropy type has the dimension
  \begin{equation*}
    \dim G -\dim G_x +(\LT')^L= 24-6+2 = 20,
  \end{equation*}
  where $L\subset G$ is given by the group with Lie algebra
  $\LT \oplus \langle X_1, X_2, Y_1, Y_2\rangle$.
  
  If $x\in (V_0)_\tau$ and $\abs{x_\alpha} = \abs{x_\beta}$, then
  $\LG_x \cong \mathfrak{so}(5)$ and the stratum an $x$ generically
  corresponds to a relative equilibrium contained in an isotropy type
  stratum of dimension
   \begin{equation*}
    \dim G -\dim G_x +(\LT')^{\SU(4)\times S^1}= 24-10+1 = 15,
  \end{equation*}
  where we identify $\SU(4)$ with the image of the above embedding
  $\SU(4) \to \SU(5)$ and $S^1$ with the group $\exp(t\cdot
  \diag(1,1,1,1,-4))$, $t\in \R$.
\end{example}
In the next example, we start again with a given group $G$, the
$G$-representation $V^c$ and a given subspace $V_0\subset V^c$. The
example illustrates the proceeding to compute the isotropy Lie algebras:
\begin{example}\label{ex:combined}
  Consider $G=\SU(6)$ and let $V^c$ be the complexified adjoint
  representation of $G$. Again, we identify $\LT$ and $\LT^*$ with the
  subspace of $\R^6$ of the vectors whose entries sum up to $0$.
  Let $V_0$ the sum of weight spaces
  corresponding to the three linearly independent roots
  $\alpha:=(1,-1,0,0,0,0)$, $\beta:=(0,0,1,-1,0,0)$, and
  $\gamma:=(0,0,0,0,1,-1)$. 

  Note that the affine span of $\alpha$, $\beta$, and $\gamma$
  contains the point
  \begin{equation*}
    \nu := \frac 1 3(1,-1,1,-1,1,-1).
  \end{equation*}
  Its coadjoint isotropy subgroup is given by
  $G_\nu\simeq \SU(3) \times \SU(3) \times S^1$, where the first
  $\SU(3)$-factor is given by the matrices whose non-vanishing entries
  have odd row and column numbers, the second factor is given by the
  matrices whose non-vanishing entries have even row and column
  numbers and the $S^1$-component is given by matrices of the form
  $\diag(\theta, \bar \theta, \theta, \bar \theta, \theta, \bar
  \theta)$, $\abs{\theta}=1$.

  Since $(\LT^*)^{G_\nu} = \langle \nu \rangle$ intersects the affine
  span of $\alpha$, $\beta$ and $\gamma$ in the single point $\nu$ and
  is orthogonal to it, all points of $(V_0)_\tau$ have isotropy
  subgroups contained in $G_\nu$.

  For all $x \in (V_0)_\tau$ with pairwise different $\abs{x_\alpha}$,
  $\abs{x_\beta}$, and $\abs{x_\gamma}$ the coadjoint isotropy
  subgroup $G_\mu$ of $\mu := \J(x)$ coincides with $T$ and hence
  $\LG_x = \LT_x$.  The corresponding stratum has the dimension
  \begin{equation*}
    35-2+3=36.
  \end{equation*}
  To compute the isotropy Lie algebras of points of $(V_0)_\tau$ that
  have at least two equal values among $\abs{x_\alpha}$,
  $\abs{x_\beta}$, and $\abs{x_\gamma}$, we consider the action of the
  root vectors of $G_\nu\simeq \SU(3) \times \SU(3) \times
  S^1$. Consider $\{X_{1,i},Y_{1,i},X_{2,i},Y_{2,i}\}_{i=1,2,3}$,
  where $X_{1,i}, Y_{1,i}$ are the images of the corresponding root
  vectors $X_i, Y_i$ of $\SU(3)$ (see Example~\ref{ex:pyramid}) with
  respect to the embedding $A \mapsto \bar A_1$ of $\su(3)$ into the
  first $\su(3)$-factor of $\LG_\nu$ given by
  $(\bar A_1)_{2i-1, 2j-1} = (A)_{i,j}$ . Similarly $X_{2,i}, Y_{2,i}$
  denote the image of the corresponding root vectors under the
  embedding of $\su(3)$ the second $\su(3)$-factor of $\LG_\nu$ given
  by $A \mapsto \bar A_2$ with $(\bar A_2)_{2i,2j} = (A)_{i,j}$.

  Now, we consider the neighbors of the weights of $V_0$ with respect
  to the roots of $G_\nu$, which are given by the positive roots
  \begin{align*}
    \rho_{1,1} = (1,0,-1,0,0,0),\quad &\rho_{1,2} = (0,0,1,0,-1,0), \quad
     &\rho_{1,3} = (1,0,0,0,-1,0),\\
    \rho_{2,1} = (0,1,0,-1,0,0),\quad &\rho_{2,2} = (0,0,0,1,0,-1), \quad
    &\rho_{2,3} = (0,1,0,0,0,-1).                                    
  \end{align*}
  and their additive inverses. There are no common neighbors of all
  three of $\alpha$, $\beta$, and $\gamma$. For each pair of weights
  of $V_0$ there are two weights which are neighbors of both. For
  example, the weights $\alpha:=(1,-1,0,0,0,0)$ and
  $\beta:=(0,0,1,-1,0,0)$ have the common neighbors $(0,-1,1,0,0,0)$
  and $(1,0,0,-1,0,0)$. Here, the difference vectors are given by the
  roots $\pm \rho_{1,1}$ and $\pm \rho_{2,1}$, which form a subroot
  system isomorphic to $A_1 \times A_1$, the root system of
  $\SU(2) \times \SU(2)$. Each of the corresponding root vectors
  $X_{1,1}$, $Y_{1,1}$, $X_{2,1}$ and $Y_{2,1}$ acts trivially on the
  weight space $\C_\gamma$. Since for $x \in (V_0)_\tau$ with
  $\abs{x_\alpha} = \abs{x_\beta} \ne \abs{x_\gamma}$ with
  $\mu := \J(x)$ these four roots form the root system of $G_\mu$, the
  $M_{ \rho_{1,1}} \oplus M_{\rho_{2,1}}$ part of $\LG_x$ is given by
  the solutions of a system of linear equations similar to that in
  Example~\ref{ex:cube}. Thus $\LG_x$ contains a unique subalgebra
  isomorphic to $\su(2)$ and the rank of $\LG_x$ is $2$. Thus $\LG_x$
  is isomorphic to $\su(2) \oplus \R$. Since all $x \in (V_0)_\tau$
  with $\abs{x_\alpha}= \abs{x_\beta}$ have the same isotropy type, we
  obtain that the corresponding stratum has the dimension
  \begin{equation*}
    35 - 4 + 2=33.
  \end{equation*}
  Moreover, since the Weyl group permutes the weight spaces, the
  points with $\abs{x_\alpha} = \abs{x_\gamma} \ne \abs{x_\beta}$ or
  $\abs{x_\alpha} \ne\abs{x_\gamma} =\abs{x_\beta}$ have a point with
  $\abs{x_\alpha} = \abs{x_\beta} \ne \abs{x_\gamma}$ in its $G$-orbit
  and are hence contained in the same isotropy stratum.

  The last case is $\abs{x_\alpha}=
  \abs{x_\beta}=\abs{x_\gamma}$. Then $\J(x) = \nu$ and we have to
  consider all roots of $G_\nu$. As pointed out above, each two of the
  weights $\alpha$, $\beta$ and $\gamma$ have two common neighbors,
  but there is no common neighbor of all three weights. In addition,
  $\LG_\nu\otimes \C$ is a direct sum of $\LT\otimes \C$ and the three
  subspaces
  $\mathfrak k:=\langle X_{1,1}, Y_{1,1},X_{2,1}, Y_{2,1}\rangle$,
  $\mathfrak l:=\langle X_{1,2}, Y_{1,2},X_{2,2}, Y_{2,2}\rangle$ and
  $\mathfrak m:\langle X_{1,3}, Y_{1,3},X_{2,3}, Y_{2,3}\rangle$. The
  matrices in $\mathfrak k$ map the weight spaces of $\alpha$ and of
  $\beta$ the weight spaces of their common neighbors and vanish on
  $\C_\gamma$. Similarly, the image of the elements of $\mathfrak l$
  applied to $V_0$ is contained in the span of the weight spaces
  corresponding to the common neighbors of $\beta$ and $\gamma$, and
  the image of $V_0$ under the action of the elements of $\mathfrak m$
  is contained in the span of the weight spaces corresponding to the
  common neighbors of $\alpha$ and $\gamma$. Hence a vector
 \begin{equation*}
   Z = Z_\LT + Z_{\mathfrak k} + Z_{\mathfrak l} + Z_{\mathfrak m}
   \in \LT \oplus \mathfrak k \oplus \mathfrak l \oplus \mathfrak m
 \end{equation*}
 satisfies $Zx = 0$ iff this holds for each of the components $Z_\LT$,
 $Z_{\mathfrak k}$, $Z_{\mathfrak l}$, $Z_{\mathfrak m}$. In addition,
 \begin{equation*}
   \LG_\nu = \LT \oplus (\mathfrak k \cap   \LG_\nu)
   \oplus (\mathfrak l \cap   \LG_\nu)
   \oplus (\mathfrak m \cap  \LG_\nu),
 \end{equation*}
 and the solution set of $\xi x = 0$ within each of the subspaces
 $(\mathfrak k \cap \LG_\nu)$,$(\mathfrak l \cap \LG_\nu)$,
 $(\mathfrak m \cap \LG_\nu)$ corresponds to a linear system similar
 to the one in Example~\ref{ex:cube}.  Hence $\LG_x$ is an
 $8$-dimensional Lie algebra of rank $2$ and contains three Lie
 subalgebras isomorphic to $\su(2)$.

 Obviously, $\LT_x$ is the 2-dimensional subspace of $\LT$ given by
 elements of the form $(a,a, b,b,c,c)$, $a,b,c \in \R, a+b+c =
 0$. Moreover, consider the positive roots
 $\rho_{1,1} =(1,0,-1,0,0,0)$ and $\rho_{2,1}=(0,1,0,-1,0,0)$ that
 correspond to the subspace $\mathfrak k$. If we evaluate these roots
 at elements of $\LT_x$, we obtain the result $(a-b)$ for both of
 them. Thus $(a,a, b,b,c,c) \mapsto (a-b)$ is a root of $G_x^\circ$
 with a root vector contained in $\mathfrak k$. Similarly, the
 intersections $\mathfrak l \cap \LG_x$ and $\mathfrak m \cap \LG_x$
 correspond to the pairs of roots that map elements of the above form
 to $\pm(a-c)$ and $\pm(b-c)$ respectively. Thus $\LG_x$ is isomorphic
 to $\su(3)$.

  The corresponding isotropy stratum has the dimension
  \begin{equation*}
    35-8+2=29.
  \end{equation*}
\end{example}

We know briefly sketch an approach to compute the isotropy types on
the Lie algebra level within the branches of relative equilibria that
we obtain from Theorem~\ref{thm:main-thesis} in general:

We have to consider all sums of weight spaces corresponding to a
linearly independent subset $S$ of the weights of $V^c$ such that each
subset $S_i$ of $S$ consisting of the weights of $S$ contained in an
irreducible $G$-subrepresentation is maximal within the sum of its
affine span and the underlying subspaces of the affine spans of the
other subsets $S_j \subset S$ of this kind. To determine the types of
the isotropy Lie algebras, it is reasonable to start with the complex
1-dimensional spaces. Then we increase the complex dimension step by
step. The advantage of this proceeding is that we can often reuse
calculations for lower dimensional subspaces as we have seen for
instance in Example~\ref{ex:combined}.

Since the isotropy subgroup of $x \in V_c$ coincides with the
intersection of the isotropy subgroups of the components of $x$ in the
$G$-irreducible $\der^2h(0)|_{V_c}$-eigenspaces, w.~l.~o.~g. we assume
in the following that $V^c$ is $G$-irreducible.

Obviously, all non-zero points within a weight space $\C_\alpha$ are
of the same isotropy type. They even have the same isotropy Lie
algebras: For $0\ne x \in \C_\alpha$, the complexified isotropy Lie
algebra $\C \otimes \LG_x$ is given by the span of $\LT_x \otimes \C$
and the root vectors $Z$ corresponding to roots $\rho$ of $\LG_\alpha$
with $Z \left(\C_\alpha\right) =0$, which holds iff $\alpha+\rho$ is
not a weight of $V^c$. If $\rho$ is a root of $\LG_\alpha$, then $\alpha+\rho$
is a weight of $V^c$ iff this holds for $\alpha-\rho$. Recall that we
denote the weight space in $\LG\otimes \C$ corresponding to a root
$\rho$ by $L_\rho$ and the real space
$(L_\rho\oplus L_{-\rho}) \cap \LG$ by $M_\rho=M_{-\rho}$.  Then
$\LG_x$ coincides with the sum of $\LT_x$ and the spaces $M_\rho$ such
that $\rho$ is a root of $\LG_\alpha$ and $\alpha + \rho $ is not a
weight of $V^c$.

To compute the isotropy algebras of the elements of $(V_0)_\tau$ for
$V_0= \sum_{\alpha \in S}\C_\alpha$, we have to consider all coadjoint
isotropy subgroups $K$ of elements of $\LT^*$ with the following
property (see Remark~\ref{rem:characterization_L_part2}): $S$ splits
into subsets $S_i$ whose convex hulls contain a single point of
$(\LT^*)^K$, which is an inner point if S has more than one point,
and whose affine spans are orthogonal to $(\LT^*)^K$. To compute the
Lie algebra $\LG_x = \mathfrak k_x$ we only need to intersect the
isotropy Lie algebras of the components $x_i$ of $x$ within
$ \sum_{\alpha \in S_i}\C_\alpha$. If the partition contains more than
one set, then these isotropy Lie algebras have been determined in
previous steps.

Thus to complete the computation of the types of isotropy Lie algebras
of the elements of $(V_0)\tau$, we only have to calculate the groups
$K$ of this type such that $((\LT)^*)^K$ intersects the convex hull of
$S$ orthogonally in a single point $\nu$. If there is such a group $K$, then it is
maximal within all coadjoint isotropy subgroups of elements of $\LT^*$
that intersect the affine span of $S$ orthogonally. Thus there is at
most one such group.

Moreover, the set of points in $V_0$ with momentum $\nu$ consists of a
single $T$-orbit. Hence to compute the types of their isotropy Lie
algebras, we can choose a single $x$ with $\J(x) = \nu$ and compute
$\LG_x$.

To compute $\LG_x$, we proceed similarly as in our examples: We split
$\LG_\nu$ into $\LT$ and the span of spaces $M_\rho$ for the roots
$\rho$ of $G_\nu$. No two weights of $V_0$ are neighbors of each
other: If the difference of two weights of $V_0$ was a root, then
their affine span would contain a Weyl reflection pair and hence this
would hold for $S$. Thus $(\LG_\nu\otimes \C)_x$ coincides with the
sum of $(\LT \otimes T)_x$ and the annihilator of $x$ within the span
of the spaces $L_\rho$ for the roots $\rho$ of $G_\nu$. Therefore
$\LG_x$ is given by the sum of $\LT_x$ and the space of solutions
$\xi$ in the span of the spaces $M_\rho$ of the equation $\xi x=0$.
To determine the solutions $\xi$, we choose a basis of each space
$M_\rho$. (Since the each of the spaces $M_\rho$ is a real vector
space of dimension 2, we have to choose 2 vectors for each root $\rho$
of $\LG_\mu$. In our examples, we have chosen vectors of the form
$X_i - Y_i$ and $\I(X_i+Y_i)$, where $X_i$ and $Y_i$ are root vectors
for $\rho$ and $-\rho$ respectively.) Then together, these vectors
form a basis $Z_1, \dots, Z_{2k}$ of the span of the spaces
$M_\rho$. We then consider the matrix $A$ with columns $(Z_ix)$. Then
$\ker A \subset \R^{2k}$ corresponds to the space of solutions $\xi$
of $\xi x=0$ within the span of the spaces $M_\rho$.

As we have seen in Example~\ref{ex:combined}, it can happen that the
linear system decouples into subsystems that have been solved in order
to determine the isotropy Lie algebras of the main stratum of a
subspace of $V_0$. This is another reason for our proceeding.

\section*{Acknowledgments}
I thank Reiner Lauterbach for reading a preliminary version of this article and
giving helpful advice.

\printbibliography
\end{document}